\documentclass{amsart}
\usepackage[all]{xy} 
\usepackage{xspace, graphicx, epsfig, amssymb, amsmath, amsthm, stmaryrd}
\usepackage{setspace}

\renewcommand{\sl}{sl}
\newcommand{\so}{o}
\newcommand{\T}{\mathbb T}
\newcommand{\Tg}{\mathbb T_{\mathfrak g}}
\renewcommand{\sp}{sp}

\renewcommand{\gg}{\mathfrak{g}}
\newcommand{\hh}{\mathfrak{h}}
\newcommand{\bb}{\mathfrak{b}}
\newcommand{\nn}{\mathfrak{n}}
\renewcommand{\dim}{\operatorname{dim}}
\newcommand{\CC}{\mathbb{C}}
\newcommand{\ZZ}{\mathbb{Z}}
\newcommand{\limarr}{\underrightarrow{\mathrm{lim}} \;}
\newcommand{\Aut}{\operatorname{Aut}}
\newcommand{\Ann}{\operatorname{Ann}}
\newcommand{\soc}{\operatorname{soc}}
\newcommand{\rad}{\operatorname{rad}}
\newcommand{\Hom}{\operatorname{Hom}}

\newtheorem{thm}{Theorem}[section]
\newtheorem{lemma}[thm]{Lemma}
\newtheorem{prop}[thm]{Proposition}
\newtheorem{cor}[thm]{Corollary}

\theoremstyle{definition}
\newtheorem{defn}[thm]{Definition}
\newtheorem{rem}[thm]{Remark}

\numberwithin{equation}{section}

\author{Elizabeth Dan-Cohen}
\address{Jacobs University Bremen, Campus Ring 1, 28759 Bremen, Germany}
\email{elizabeth.dancohen@gmail.com}

\author{Ivan Penkov}
\address{Jacobs University Bremen, Campus Ring 1, 28759 Bremen, Germany}
\email{i.penkov@jacobs-university.de}

\author{Vera Serganova}
\address{Department of Mathematics, University of California Berkeley, Berkeley CA 94720, USA}
\email{serganov@math.berkeley.edu}

\title{A Koszul category of representations of finitary Lie algebras}

\subjclass[2000]{17B65, 17B10, 16G10}

\date{May 10, 2012}

\keywords{Koszul duality, finitary Lie algebras}

\begin{document}

\begin{abstract}
We find for each simple finitary Lie algebra $\gg$ a category $\Tg$ of integrable modules in which the tensor product of copies of the natural and conatural modules are injective.  The objects in $\Tg$ can be defined as the finite length absolute weight modules, where by absolute weight module we mean a module which is a weight module for every splitting Cartan subalgebra of $\gg$.  The category $\Tg$ is Koszul in the sense that it is antiequivalent to the category of locally unitary finite-dimensional modules over a certain direct limit of finite-dimensional Koszul algebras.  We describe these finite-dimensional algebras explicitly.  We also prove an equivalence of the categories $\T_{\so(\infty)}$ and $\T_{\sp(\infty)}$ corresponding respectively to the orthogonal and symplectic finitary Lie algebras $\so(\infty)$, $\sp(\infty)$.
\end{abstract}

\maketitle

\section{Introduction}

The classical simple complex Lie algebras $\sl(n)$, $\so(n)$, $\sp(2n)$ have several natural infinite-dimensional versions. In this paper we concentrate on the ``smallest possible" such versions: the direct limit Lie algebras $\sl(\infty):=\limarr (\sl(n))_{n\in\mathbb{Z}_{>2}}$, $\so(\infty):=\limarr ( \so(n))_{n\in\mathbb{Z}_{\geq 3}}$, $\sp(\infty):=\limarr (\sp(2n))_{n\in\mathbb{Z}_{\geq 2}}$. 
 From a traditional finite-dimensional point of view, these Lie algebras are a suitable language for various stabilization phenomena, for instance stable branching laws as studied by R.~Howe, E.-C.~Tan and J.~Willenbring \cite{HTW}. The direct limit Lie algebras $\sl(\infty)$, $\so(\infty)$, $\sp(\infty)$ admit many characterizations: for instance, they represent (up to isomorphism) the simple finitary (locally finite) complex Lie algebras \cite{B, BSt}.  Alternatively, these Lie algebras are the only three locally simple locally finite complex Lie algebras which admit a root decomposition \cite{PStr}.

Several attempts have been made to build a basic representation
theory for $\gg=\sl(\infty)$, $\so(\infty)$, $\sp(\infty)$. 
As the only simple finite-dimensional representation of $\gg$ is the
trivial one, one has to study infinite-dimensional
representations. On the other hand, it is still possible to study
representations which are close analogs of finite-dimensional
representations.
Such a representation should certainly be integrable, i.e.\ it should
be isomorphic to a direct sum of finite-dimensional representations
when restricted to any simple finite-dimensional subalgebra.

The first phenomenon one encounters when studying integrable
representations of $\gg$ is that they are not in general semisimple.
This phenomenon has been studied in \cite{PStyr} and
\cite{PS}, but it had not previously been placed within a known more general framework for non-semisimple categories. 
The main purpose of the present paper is to show that the notion of Koszulity
for a category of modules over a graded ring, as defined by
A. Beilinson, V. Ginzburg and W. Soergel in \cite{BGS}, provides an 
excellent tool for the study of integrable representations of  $\gg=\sl(\infty)$, $\so(\infty)$, $\sp(\infty)$.

In this paper
we introduce the category $\Tg$ of 
tensor $\gg$-modules. The objects of   $\Tg$ are defined at first by the equivalent abstract conditions of Theorem \ref{th1}. Later we show in Corollary
\ref{newlab} that the objects of $\Tg$ are nothing but finite length
submodules of a direct sum of several copies of the tensor algebra $T$
of the natural and conatural representations. In the
finite-dimensional case, i.e. for $\sl(n)$, $\so(n)$, or $\sp(2n)$, the appropriate 
 tensor algebra is a
cornerstone of the theory of finite-dimensional representations (Schur-Weyl duality, etc.).
  In the infinite-dimensional case, the tensor algebra $T$ was studied by Penkov and
K.~Styrkas in \cite{PStyr}; nevertheless its indecomposable direct
summands were not understood until now from a categorical point of
view. 

We prove that these indecomposable modules 
 are precisely
the indecomposable injectives in the category $\Tg$.  Furthermore, the category $\Tg$ is Koszul in the following sense: $\Tg$ is
antiequivalent to the category of locally unitary finite-dimensional
modules over an algebra $\mathcal{A}_\gg$ which is a direct limit of
finite-dimensional Koszul algebras (see Proposition \ref{equiv} and Theorem
\ref{koszulity}). 

Moreover, we prove in Corollary \ref{selfduality} that for $\gg=\sl(\infty)$ 
the Koszul dual algebra $(\mathcal{A}_\gg^!)^\mathrm{opp}$ is 
isomorphic to $\mathcal{A}_\gg$. This together with the main result of \cite{PStyr} allows us to give
an explicit formula for the Ext group between any two simple objects of $\Tg$ when
$\gg=\sl(\infty)$.
For the cases of  $\gg=\so(\infty)$ and $\gg=\sp(\infty)$ we discover
another interesting fact:
the algebras $\mathcal{A}_{\so(\infty)}$ and
$\mathcal{A}_{\sp(\infty)}$ are isomorphic. This yields an 
equivalence of categories $\T_{\so(\infty)}\simeq \T_{\sp(\infty)}$, which is
Corollary \ref{isomosp}. 

In summary, the results of the present paper show how the
non-semisimplicity of tensor 
modules arising from the limit process $n\rightarrow\infty$ falls strikingly into the general Koszul pattern discovered by Beilinson, Ginzburg and Soergel.  This enables us to uncover the structure of the category of tensor representations of $\gg$.

As a last remark, let us point out that the categories $\Tg$ for $\gg= \sl(\infty)$, $\so(\infty)$, and
$\sp(\infty)$ will likely prove useful for the categorification
of important classical theories, for instance the
boson-fermion correspondence. 

\subsection*{Acknowledgements}
We thank Igor Frenkel for his active and supportive interest in our
work, and Alexandru Chirvasitu for pointing out an inaccuracy in the
previous version of the paper.
All three authors gratefully acknowledge funding through DFG grants PE 980/2-1 and PE 980/3-1 (DFG SPP1388).  Vera Serganova acknowledges partial support through NSF grant 0901554.

\section{Preliminaries} \label{prelim}

The ground field is $\CC$.   By $S_n$ we denote the $n$-th symmetric group, and by $\CC[S_n]$ its group algebra.  The sign $\otimes$ stands for $\otimes_\CC$.  We denote by $(\, \cdot \,)^*$ the algebraic dual, i.e.\ $\Hom_{\CC} (\, \cdot \,, \CC)$.

Let $\gg$ be one of the infinite-dimensional simple finitary Lie algebras, $\sl(\infty)$, $\so(\infty)$, or $\sp(\infty)$.  Here $\sl(\infty)=\limarr \sl(n)$, $\so(\infty)=\limarr \so(n)$, $\sp(\infty)=\limarr \sp(2n)$, where in each direct limit the inclusions can be chosen as ``left upper corner" inclusions.  We consider the ``exhaustion" $\gg = \limarr \gg_n$ to be fixed, taking $\gg_n = \sl(n)$ for $\gg = \sl(\infty)$, $\gg_n = \so(2n)$ or $\gg_n = \so(2n+1)$ for $\gg = \so(\infty)$, and $\gg_n = \sp(2n)$ for $\gg = \sp(\infty)$.  By $G_n$ we denote the adjoint group of $\gg_n$  It is clear that $\{ G_n \}$ forms a direct system and defines an ind-group $G = \limarr G_n$.  As mentioned in the introduction, the Lie algebras $\sl(\infty)$, $\so(\infty)$, and $\sp(\infty)$ admit several equivalent intrinsic descriptions, see for instance \cite{B, BSt, PStr}.  

It is clear from the definition of $\gg = \sl(\infty)$, $\so(\infty)$, $\sp(\infty)$ that the notions of semisimple or nilpotent elements make sense: an element $g \in \gg$ is \emph{semisimple} (respectively, \emph{nilpotent}) if $g$ is semisimple (resp., nilpotent) as an element of $\gg_n$ for some $n$.  In \cite{NP, DPS}, Cartan subalgebras of $\gg$ have been studied.  In the present paper we need only the notion of a \emph{splitting Cartan subalgebra} of $\gg$: this is a maximal toral (where \emph{toral} means consisting of semisimple elements) subalgebra $\hh \subset \gg$ such that $\gg$ is an $\hh$-weight module, i.e.
$$\gg = \bigoplus_{\alpha \in \hh^*} \gg^\alpha,$$
where $\gg^\alpha = \{ g \in \gg \mid [h,g] = \alpha(h) g \textrm{ for all } h \in \hh \}$.  The set $\Delta := \{ \alpha \in \hh^* \setminus \{ 0 \} \mid \gg^\alpha \neq 0 \}$ is the \emph{set of $\hh$-roots} of $\gg$.  More generally, if $\hh$ is a splitting Cartan subalgebra of $\gg$ and $M$ is a $\gg$-module, $M$ is an \emph{$\hh$-weight module} if
$$M = \bigoplus_{\alpha \in \hh^*} M^\alpha,$$
where $M^\alpha := \{ m \in M \mid h \cdot m = \alpha(h)m \textrm{ for all } h \in \hh \}$.

By $V$ we denote the natural representation of $\gg$; that is, $V = \limarr V_n$, where $V_n$ is the natural representation of $\gg_n$.  We set also $V_* = \limarr V_n^*$; this is the conatural representation of $\gg$.  For $\gg = \so(\infty)$, $\sp(\infty)$, $V \simeq V_*$, whereas $V \not\simeq V_*$ for $\gg = \sl(\infty)$.  Note that $V_*$ is a submodule of the algebraic dual $V^* = \operatorname{Hom}_\CC (V, \CC)$ of $V$.  Moreover, $\gg \subset V \otimes V_*$, and $\sl(\infty)$ can be identified with the kernel of the contraction $\phi : V \otimes V_* \rightarrow \CC$, while
\begin{align*}
\gg & \simeq \Lambda^2 (V) \subset V \otimes V = V \otimes V_* & \textrm{for } \gg = \so(\infty), \\
\gg & \simeq S^2 (V) \subset V \otimes V = V \otimes V_* & \textrm{for } \gg = \sp(\infty).
\end{align*} 

Let $\tilde G$ be the subgroup of $\Aut V$ consisting of those
automorphisms for which the induced automorphism of $V^*$ restricts to
an automorphism of $V_*$.  Then clearly $G \subset \tilde{G} \subset
\Aut \gg$, and moreover $\tilde G = \Aut \gg$ for $\gg = \so(\infty)$,
$\sp(\infty)$ \cite[Corollary 1.6 (b)]{BBCM}.
For $\gg = \sl(\infty)$, the group $\tilde G$ has index $2$ in $\Aut \gg$: the quotient $\Aut \gg / \tilde G$ is represented by the automorphism
\begin{equation*}
g \mapsto -g^t
\end{equation*}
for $g \in \sl(\infty)$ \cite[Corollary 1.2 (a)]{BBCM}.

It is essential to recall that if $\gg = \sl(\infty)$, $\sp(\infty)$, all splitting Cartan subalgebras of $\gg$ are $\tilde{G}$-conjugate, while there are two $\tilde{G}$-conjugacy classes for $\gg = \so(\infty)$.  One class comes from the exhaustion of $\so(\infty)$ as $\limarr \so(2n)$, and the other from the exhaustion of the form $\limarr \so(2n+1)$.  For further details we refer the reader to \cite{DPS}.  Here are the explicit forms of the root systems of $\gg$:
\begin{align*}
&\{ \epsilon_i - \epsilon_j \mid i \neq j \in \ZZ_{>0} \} &\textrm{for } \gg = \sl(\infty) \textrm{, } \gg_n = \sl(n), \\
&\{ \pm \epsilon_i \pm \epsilon_j \mid i \neq j \in \ZZ_{>0} \} \cup \{ \pm 2 \epsilon_i \mid i \in \ZZ_{>0} \} &\textrm{for } \gg = \sp(\infty)  \textrm{, } \gg_n = \sp(2n), \\
&\{ \pm \epsilon_i \pm \epsilon_j \mid i \neq j \in \ZZ_{>0} \} &\textrm{for } \gg =  \so(\infty) \textrm{, } \gg_n= \so(2n), \\
&\{ \pm \epsilon_i \pm \epsilon_j \mid i \neq j \in \ZZ_{>0} \}  \cup \{ \pm \epsilon_i \mid i \in \ZZ_{>0} \} &\textrm{for } \gg = \so(\infty)  \textrm{, } \gg_n = \so(2n+1). 
\end{align*}
Our usage of $\epsilon_i \in \hh^*$ is compatible with the standard usage of $\epsilon_i$ as a linear function on $\hh \cap \gg_n$ for all $n > i$.

In the present paper we study integrable $\gg$-modules $M$ for $\gg\simeq \sl(\infty)$,
$\so(\infty)$, $\sp(\infty)$. By definition, a $\gg$-module $M$ is
\emph{integrable} if $\dim \{m, g\cdot m, g^2\cdot m, ...\}<\infty$
for all $g\in\gg$, $m \in M$. 
More generally, if $M$ is any $\gg$-module, the set $\gg [ M ]$ of $M$-locally finite elements in $\gg$, that is
$$\gg [ M ] := \{ g \in \gg \mid \dim \{ m , g \cdot m, g^2 \cdot m , \ldots \} < \infty \textrm{ for all } m \in M \},$$
is a Lie subalgebra of $\gg$.  This follows from the analogous fact for finite-dimensional Lie algebras, discovered and rediscovered by several mathematicians \cite{GQS, F, K}.  We refer to $\gg [ M ]$ as the \emph{Fernando-Kac subalgebra} of $M$.

By $\gg$-mod we denote the category of all $\gg$-modules, and  following the notation of \cite{PS}, we let $\operatorname{Int}_\gg$ denote the category of integrable $\gg$-modules.  For a fixed splitting Cartan subalgebra $\hh \subset \gg$, denote by $\operatorname{Int}^\mathrm{wt}_{\gg,\hh}$ the full subcategory of $\operatorname{Int}_\gg$ whose objects are $\hh$-weight modules.  One has the functor
$$\Gamma_\gg : \gg\textrm{-mod} \rightarrow \operatorname{Int}_\gg$$
which takes an arbitrary $\gg$-module to its largest integrable submodule,
as well as the functor
$$\Gamma^\mathrm{wt}_\hh : \operatorname{Int}_\gg \rightarrow \operatorname{Int}^\mathrm{wt}_{\gg,\hh}$$ 
which takes an integrable module $M$ to its largest $\hh$-weight submodule.

\section{The category $\Tg$} \label{catsec}

If $\gamma \in \Aut \gg$ and $M$ is a $\gg$-module, let $M^\gamma$ denote the $\gg$-module twisted by $\gamma$: that is, $M^\gamma$ is equal to $M$ as a vector space, and the $\gg$-module structure on $M^\gamma$ is given by $\gamma(g)\cdot m$ for $m \in M^\gamma$ and $g \in \gg$.

\begin{defn} 
\begin{enumerate}
\item A $\gg$-module $M$ is called an \emph{absolute weight module} if $M$ is an $\hh$-weight module for every splitting Cartan subalgebra $\hh\subset\gg$.
\item A $\gg$-module $M$ is called \emph{$\tilde{G}$-invariant} if for any $\gamma\in \tilde{G}$ there is a $\gg$-isomorphism $M^\gamma\simeq M$.
\item A subalgebra of $\gg$ is called \emph{finite corank} if it contains the commutator subalgebra of the centralizer of some finite-dimensional subalgebra of $\gg$. 
\end{enumerate}
\end{defn}

\begin{prop} \label{int}
Any absolute weight $\gg$-module is integrable.
\end{prop}

\begin{proof}
Let $M$ be an absolute weight $\gg$-module.  Every semisimple element $h$ of $\gg$ lies in some splitting Cartan subalgebra $\hh$ of $\gg$, and since $M$ is an $\hh$-weight module, we see that $h$ acts locally finitely on $M$.  As $\gg$ is generated by its semisimple elements, the Fernando-Kac subalgebra $\gg [M]$ equals $\gg$, i.e.\ $M$ is integrable.
\end{proof}

We define the category of absolute weight modules as the full subcategory of $\gg$-mod whose objects are the absolute weight modules.  Proposition~\ref{int} shows that the category of absolute weight modules is in fact a subcategory of $\operatorname{Int}_\gg$.

\begin{lemma} \label{tildes}
For each $n$ one has $\tilde{G} = G \cdot \tilde{G}'_n$, where $$\tilde{G}'_n := \{ \gamma \in \tilde{G} \mid \gamma (g) = g \textrm{ for all } g \in \gg_n \}.$$
\end{lemma}

\begin{proof}
Let $\gg$ be $\so(\infty)$ or $\sp(\infty)$, and let $\gamma \in \tilde{G}$.   Fix a basis $\{ w_i \}$ of $V_n$.  There exists $\gamma'' \in G$ such that $(\gamma'')^{-1} (\gamma (w_i)) = w_i$ for all $1 \leq i \leq 2n$.  Since $\gg \subset V \otimes V$ and $\gg_n = \gg \cap (V_n \otimes V_n)$, we see that $(\gamma'')^{-1} \gamma \in \tilde{G}'_n$.

For $\gg = \sl(\infty)$, the analogous statement is as follows.  In this case one has $\gg_n = \gg \cap (V_n \otimes V_n^*)$.  Fix dual bases  $\{ w_i \}$ and $\{ w_i^* \}$ of $V_n$ and $V_n^*$, respectively.  Then for any $\gamma \in \tilde{G}$, there is an element $\gamma'' \in G$ such that $(\gamma'')^{-1} (\gamma (w_i)) = w_i$ and $(\gamma'')^{-1} (\gamma (w_i^*)) = w_i^*$ for each $1 \leq i \leq n$.  Therefore $(\gamma'')^{-1} \gamma \in \tilde{G}'_n$.
\end{proof}

\begin{thm}\label{th1}
The following conditions on a $\gg$-module $M$ of finite length are equivalent:
\begin{enumerate}
\item \label{1one} $M$ is an absolute weight module.
\item \label{1three} $M$ is a weight module for some splitting Cartan subalgebra $\hh\subset\gg$ and $M$ is $\tilde{G}$-invariant.
\item \label{1four} $M$ is integrable and $\Ann_\gg m$ is finite corank for all $m\in M$.
\end{enumerate}
\end{thm}

\begin{proof}
Let us show that \eqref{1one} implies \eqref{1four}.  We already proved in Proposition~\ref{int} that a $\gg$-module $M$ satisfying \eqref{1one} is integrable.  Furthermore, it suffices to prove that $\Ann_\gg m$ is finite corank for all $m\in M$ under the assumption that the $\gg$-module $M$ is simple.  This follows from the observation that a finite intersection of finite corank subalgebras is finite corank. 

Fix a splitting Cartan subalgebra $\hh$ of $\gg$ such that $\hh \cap \gg_n$ is a Cartan subalgebra of $\gg_n$; let $\bb = \hh \supsetplus \nn$ be a Borel subalgebra of $\gg$ whose set of roots (i.e.\ positive roots) is generated by the simple (i.e.\ indecomposable) roots.  Fix standard bases $e_i$, $h_i$, $f_i$ for the corresponding root $\sl(2)$-subalgebras.  Fix a nonzero $\hh$-weight vector $m \in M$.  

Choose a set of commuting simple roots $\{ \alpha_i \mid i \in I \}$ of $\bb$.
The set of semisimple elements $\{ h_i + e_i \mid i \in I \}$ is $\tilde G$-conjugate to the set $\{ h_i \mid i \in I \}$, and can thus be extended to a splitting Cartan subalgebra $\hh'$ of $\gg$.  Since $M$ is an absolute weight module, there is a nonzero $\hh'$-weight vector $m' \in M$.  As $M$ is simple, it must be that $m \in U(\gg) \cdot m'$.  Moreover, one has $m \in U(\gg_n) \cdot m'$ for some $n$.  For almost all $i$, $h_i$ and $e_i$ commute with $\gg_n$, in which case $m$ is an eigenvector for $h_i + e_i$.  Thus $e_i \cdot m$ is a scalar multiple of $m$.  Since $M$ is integrable, $e_i$ acts locally nilpotently, and we conclude that $e_i \cdot m = 0$ for all but finitely many $i$.  By considering the set $\{ h_i + f_i \mid i \in I \}$ in place of $\{ h_i + e_i \mid i \in I \}$, we see that $f_i \cdot m = 0$ for all but finitely many $i$, and hence $e_i \cdot m = f_i \cdot m = 0$ for all but finitely many $i \in I$.  

We now consider separately each of the three possible choices of $\gg$.  For $\gg = \sl(\infty)$, we may assume that the simple roots of $\bb$ are of the form $\{ \epsilon_i - \epsilon_{i+1} \mid i \in \ZZ_{>0} \}$.  
We first choose the set of commuting simple roots to be $\{\epsilon_{2i-1} - \epsilon_{2i} \mid i \in \ZZ_{>0} \}$ 
and obtain in this way that $e_i \cdot  m = f_i \cdot m = 0$ for almost all odd indices $i$.  By choosing the set of 
commuting simple roots as $\{\epsilon_{2i} - \epsilon_{2i+1} \mid i
\in \ZZ_{>0} \}$, we have $e_i \cdot  m = f_i \cdot m =  0$ for almost 
all even indices $i$, hence for almost all $i$.  Since it contains
$e_i$ and $f_i$ for almost all $i$, the subalgebra
$\Ann_\gg m$ is a finite corank subalgebra of $\gg = \sl(\infty)$.  

For $\gg = \so(\infty)$, one may assume that the set of simple roots of $\gg$ is $\{ -\epsilon_1 - \epsilon_2 \} \cup \{ \epsilon_i - \epsilon_{i+1} \mid  i \in \ZZ_{>0} \}$, and one considers in addition to the above two
sets of positive roots the sets $\{ -\epsilon_{2i} - \epsilon_{2i+1}
\mid i \in \ZZ_{>0} \}$ and $\{ -\epsilon_{2i-1} - \epsilon_{2i} \mid i
\in \ZZ_{>0} \}$.  For $\gg = \sp(\infty)$, the set of simple roots can be chosen as $\{ -2\epsilon_1 \} \cup \{ \epsilon_i - \epsilon_{i+1} \mid  i \in \ZZ_{>0} \}$, and one considers in addition to the two sets of positive roots for $\sl(\infty)$ the sets  $\{ -2\epsilon_{2i
  +1} \mid i \in \ZZ_{>0} \}$ and $\{ -2\epsilon_{2i} \mid i \in
\ZZ_{>0} \}$.  The remainder of the argument is then the same as for $\sl(\infty)$.

Next we prove that \eqref{1four} implies \eqref{1three}.  

We first show that a $\gg$-module $M$ satisfying \eqref{1four} is a weight
module for some splitting Cartan subalgebra $\hh \subset \gg$.  Fix a
finite set $\{ m_1, \ldots, m_s \}$ of generators of $M$. Let $\gg'_n$ be the commutator subalgebra of the centralizer in $\gg$ of $\gg_n$.  There exists a finite corank subalgebra that annihilates $m_1, \ldots, m_s$, and hence $\gg'_n$ annihilates $m_1, \ldots, m_s$ for some $n$. Let $\hh'_n$
be a splitting Cartan subalgebra of $\gg'_n$. Obviously $M$ is
semisimple over $\hh'_n$. One can find $k$ and a Cartan
subalgebra $\hh_k\subset\gg_k$ such that $\hh=\hh'_n+\hh_k$ is a
splitting Cartan subalgebra of $\gg$. (If $\gg=\so(\infty)$ or
$\sp(\infty)$ one can choose $k=n$; if $\gg=\sl(\infty)$, one can set $k=n+1$). Since $M$ is integrable, $M$ is semisimple over $\hh_k$. Hence $M$ is semisimple over $\hh$. 

To finish the proof that \eqref{1four} implies \eqref{1three}, we need to show that $M$ is $\tilde{G}$-invariant.  For each $n$ one has 
$\tilde{G} = G \cdot \tilde{G}'_n$ by Lemma~\ref{tildes}. 
Fix $\gamma \in \tilde{G}$ and $m \in M$.  Then for some $n$, the vector $m$ is fixed by $\gg'_n$.   We choose a decomposition $\gamma = \gamma'' \gamma'$ so that $\gamma' \in \tilde{G}'_n$ and $\gamma'' \in G$.  We then 
set $\gamma (m) := \gamma'' (m)$, and note that the action of $G$ on $M$ is well-defined because $M$ is assumed to be integrable.  This yields a well-defined $\tilde{G}$-module structure on $M$ since, for any other decomposition $\gamma = \bar{\gamma}'' \bar{\gamma}'$ as above, one has $(\bar{\gamma}'')^{-1}\gamma'' = \bar{\gamma}' (\gamma')^{-1} \in \tilde{G}'_n \cap G = \{ \gamma \in G \mid \gamma (g) = g \textrm{ for all } g \in \gg_n \}$ which must preserve $m$.

Fix now $\gamma \in \tilde{G}$ and consider the linear operator
$$\varphi_\gamma : M^\gamma \rightarrow M, \hspace{2ex} m \mapsto \gamma^{-1} (m).$$
We claim that $\varphi_\gamma$ is an isomorphism.  For this we need to check that $g \cdot \varphi_\gamma (m) = \varphi_\gamma (\gamma (g) \cdot m)$ for any $g \in \gg$ and $m \in M$.  We have $g \cdot \varphi_\gamma (m) = g \cdot (\gamma^{-1} (m))) = \varphi_\gamma (\gamma (g \cdot \gamma^{-1} (m)))$, hence it suffices to check that $\gamma (g \cdot \gamma^{-1} (m)) = \gamma (g) \cdot m$ for every $g \in \gg$ and $m \in M$.  After choosing a decomposition $\gamma = \gamma'' \gamma'$ such that $\gamma'' \in G$ and $\gamma'$ fixes $m$, $g$ and $g \cdot m$, all that remains to check is that
\begin{equation*}
\gamma'' (g \cdot (\gamma'')^{-1} (m)) = \gamma'' (g) \cdot m
\end{equation*}
for all $g \in \gg$.  This latter equality is the well-known relation between the $G$-module structure on $M$ and the adjoint action of $G$ on $\gg$.

To complete the proof of the theorem we need to show that
\eqref{1three} implies \eqref{1one}.  What is clear is that
\eqref{1three} implies a slightly weaker statement, namely that $M$ is
a weight module for any splitting Cartan subalgebra belonging to the
same $\tilde{G}$-conjugacy class as the given splitting Cartan
subalgebra $\hh$.  For $\gg = \sl(\infty), \sp(\infty)$, this proves
\eqref{1one}, as all splitting Cartan subalgebras are conjugate under
$\tilde{G}$.  

Consider now the case $\gg=\so(\infty)$. In this case there are two $\tilde{G}$-conjugacy classes of
splitting Cartan subalgebras \cite{DPS}. Note that if $M$ is
semisimple over every Cartan subalgebra from one $\tilde{G}$-conjugacy
class, then   \eqref{1four} holds as follows from the proof of the implication
\eqref{1one} $\Rightarrow$ \eqref{1four}.  
Furthermore, the proof
that a $\gg$-module of finite length $M$ satisfying \eqref{1four} is a weight module for some
splitting Cartan subalgebra involves a choice of $\gg_n$.
For $\gg=\so(\infty)$ there are two different possible choices, namely $\gg_n=\so(2n)$ and $\gg_n=\so(2n+1)$, which in turn produce splitting Cartan subalgebras from the two $\tilde{G}$-conjugacy classes.  This shows that in each $\tilde{G}$-conjugacy class there is a splitting Cartan subalgebra of $\gg$ for which $M$ is a weight module, and hence we may conclude that \eqref{1three} implies \eqref{1one} also for $\gg= \so(\infty)$.
\end{proof}

\begin{cor} \label{twoconjugacyclasses}
Let $\gg = \so(\infty)$ and $M$ be a finite length $\gg$-module which is an $\hh$-weight module for all splitting Cartan subalgebras $\hh \subset \gg$ in either of the two $\tilde{G}$-conjugacy classes.  Then $M$ is an $\hh$-weight module for all splitting Cartan subalgebras $\hh$ of $\gg$.
\end{cor}

\begin{proof}
Let $\hh$ be a Cartan subalgebra of $\gg$, and let $M$ be a finite length $\gg$-module which is a weight module for all splitting Cartan subalgebras in the $\tilde{G}$-conjugacy class of $\hh$.  Then $M$ is integrable, by the same argument as in the proof of Proposition~\ref{int}.  Finally, \eqref{1four} holds by the same proof as that of \eqref{1one}$\Rightarrow$\eqref{1four} in Theorem~\ref{th1}.
\end{proof}

We denote by $\Tg$ the full subcategory of $\gg$-mod consisting of finite length modules satisfying the
equivalent conditions of Theorem ~\ref{th1}.
Then $\Tg$ is an abelian category and a monoidal category with respect to the usual tensor product of $\gg$-modules, and $\Tg$ is a subcategory of the category of absolute weight modules.  In addition, for $\gg = \sl(\infty)$, $\Tg$ has an  
involution
$$( \, \cdot \, )_* : \Tg \rightarrow \Tg,$$
which one can think of as ``restricted dual."  Indeed, in this case any outer automorphism $w \in \Aut \gg$ induces the autoequivalence of categories
\begin{align*}
w_\gg : \Tg & \rightarrow \Tg \\
M & \mapsto M^w.
\end{align*}
Since, however, any object of $\Tg$ is $\tilde{G}$-invariant, the functor $w_\gg$ does not depend on the choice of $w$ and is an involution, i.e.\ $w_\gg^2 = \mathrm{id}$.  We denote this involution by $(\,\cdot \,)_*$ in agreement with the fact that it maps $V$ to $V_*$.  For $\gg = \so(\infty)$, $\sp(\infty)$, we define $(\, \cdot \,)_*$ to be the trivial involution on $\Tg$.

\section{Simple objects and indecomposable injectives of $\Tg$}

Next we describe the simple objects of $\Tg$.  For this we need to recall some results about tensor representations from \cite{PStyr}.

By $T$ we denote the tensor algebra $T( V \oplus V_*)$ for $\gg = \sl(\infty)$, and $T(V)$ for $\gg = \so(\infty)$, $\sp(\infty)$.  That is, we have 
\begin{align*}
T & := \bigoplus_{p \geq 0 \textrm{, } q \geq 0} T^{p,q} &\textrm{for } \gg = \sl(\infty),\\
\intertext{and}
T & := \bigoplus_{p \geq 0} T^p & \textrm{for } \gg = \so(\infty) \textrm{, }\sp(\infty), 
\end{align*}
where $T^{p,q} := V^{\otimes p} \otimes (V_*)^{\otimes q}$ and $T^p := V^{\otimes p}$.  In addition, we set
\begin{align*}
T^{\leq r} & := \bigoplus_{p + q \leq r} T^{p,q} &\textrm{for } \gg = \sl(\infty),\\
\intertext{and}
T^{\leq r} & := \bigoplus_{p \leq r} T^p & \textrm{for } \gg = \so(\infty) \textrm{, }\sp(\infty).
\end{align*}
By a \emph{tensor module} we mean any $\gg$-module isomorphic to a subquotient of a finite direct sum of copies of $T^{\leq r}$ for some $r$.

By a \emph{partition} we mean a non-strictly decreasing finite sequence of positive integers $\mu = ( \mu_1, \mu_2, \ldots , \mu_s )$ with $\mu_1 \geq \mu_2 \geq \cdots \geq \mu_s$.  The empty partition is denoted by $0$.

Given a partition $\mu= ( \mu_1, \mu_2, \ldots , \mu_s )$ and a classical finite-dimensional Lie algebra $\gg_n$ of rank $n \geq s$, the irreducible $\gg_n$-module $(V_n)_\mu$ with highest weight $\mu$ is always well-defined.  Moreover, for a fixed $\mu$ and growing $n$, the modules $(V_n)_\mu$ are naturally nested and determine a unique simple $(\gg = \limarr \gg_n)$-module $V_\mu := \limarr (V_n)_\mu$.  For $\gg = \sl(n)$, there is another simple $\gg$-module naturally associated to $\mu$, namely $(V_\mu)_*$.

In what follows we will consider pairs of partitions for $\gg = \sl(\infty)$ and single partitions for $\gg = \so(\infty)$, $\sp(\infty)$.  Given $\lambda = (\lambda^1, \lambda^2)$ for $\gg = \sl(\infty)$, we set $\tilde V_\lambda := V_{\lambda^1} \otimes (V_{\lambda^2})_*$.  For $\gg = \so(\infty)$, $\sp(\infty)$ and for a single partition $\lambda$, the $\gg$-module $\tilde V_\lambda$ is similarly defined: we embed $\gg$ into $\sl(\infty)$ so that both the natural $\sl(\infty)$-module  
and the conatural $\sl(\infty)$-module 
are identified with $V$ as $\gg$-modules, and define $\tilde V_\lambda$ as the irreducible $\sl(\infty)$-module $V_\lambda$ corresponding to the partition $\lambda$ as defined above.  Then $\tilde{V}_\lambda$ is generally a reducible $\gg$-module.

It is easy to see that for $\gg = \sl(\infty)$,
\begin{equation}\label{triangle}
T = \bigoplus_\lambda d_\lambda \tilde V_\lambda
\end{equation}
where $\lambda = (\lambda^1, \lambda^2)$, $d_\lambda := d_{\lambda^1} d_{\lambda^2}$, and $d_{\lambda^i}$ is the dimension of the simple $S_n$-module corresponding to the partition $\lambda^i$ for $n > | \lambda^i |$.  For $\gg = \so(\infty)$, $\sp(\infty)$, Equation \eqref{triangle} also holds, with $\lambda$ taken to stand for a single partition.  Both statements follow from the obvious infinite-dimensional version of Schur-Weyl duality for the tensor algebra $T$ considered as an $\sl(\infty)$-module (see for instance \cite{PStyr}).  Moreover, according to \cite[Theorems 3.2, 4.2]{PStyr}, 
\begin{equation} \label{nabla}
\soc (\tilde V_\lambda) = V_\lambda
\end{equation}
for $\gg = \so(\infty)$, $\sp(\infty)$, while $\soc (\tilde{V}_\lambda)$ is a simple $\gg$-module for $\gg = \sl(\infty)$ \cite[Theorem 2.3]{PStyr}.  Here $\soc(\, \cdot \,)$ stands for the socle of a $\gg$-module.  We set $V_\lambda := \soc (\tilde V_\lambda)$ also for $\gg = \sl(\infty)$, so that \eqref{nabla} holds for any $\gg$.  It is proved in \cite{PStyr} that $\tilde{V}_\lambda$ (and consequently $T^{\leq r}$) has finite length.

It follows also from \cite{PStyr} that any simple tensor module is isomorphic to $V_\lambda$ for some $\lambda$.  In particular, every simple subquotient of $T$ is also a simple submodule of $T$.

For any partition $\mu = ( \mu_1, \mu_2, \ldots , \mu_s )$, we set $\# \mu := s$ and $| \mu | := \sum_{i=1}^s \mu_i$.  In the case of $\gg = \sl(\infty)$, when $\lambda = (\lambda^1, \lambda^2)$, we set $\# \lambda := \# \lambda^1 + \# \lambda^2$ and $| \lambda | := | \lambda^1 | + | \lambda^2 |$.  

We are now ready for the following lemma.

\begin{lemma} \label{nonzeromarks}
Let $\gg=\sl(\infty)$ and $\lambda=(\lambda^1,\lambda^2)$ with $\# \lambda = k > 0$. 
Then $(V_k)_{\lambda^1} \otimes (V_k^*)_{\lambda^2}$ generates
$\tilde{V}_\lambda$. 

Let $\gg = \so(\infty)$, $\sp(\infty)$, and let $\lambda$ be a partition with $\# \lambda = k > 0$. Then
the $\sl(V_{2k})$-submodule $(V_{2k})_\lambda$ of $\tilde{V}_\lambda$ generates $\tilde{V}_\lambda$. 
\end{lemma}

\begin{proof}
Let $\gg=\sl(\infty)$.  Then $M =V_{\lambda^1}\otimes (V_*)_{\lambda^2}$, 
and let $M_n:=(V_n)_{\lambda^1} \otimes (V_n^*)_{\lambda^2}$. It is easy to check that the length of $M_n$ as a $\gg_n$-module stabilizes for $n \geq k$, and moreover it coincides with the length of $M$; a formula for the length of $M$ is implied by \cite[Theorem 2.3]{PStyr}.  Hence $(V_k)_{\lambda^1} \otimes (V_k^*)_{\lambda^2}$ generates $M$.

For $\gg = \so(\infty)$, $\sp(\infty)$, one has
$M=\tilde{V}_{\lambda}$. Again, the length of $M_n:= (V_n)_\lambda$ equals the length of $M$ if $n\geq 2k$ (see Theorems 3.3 and 4.3 in \cite{PStyr}). Hence $(V_{2k})_\lambda$ generates $M$.
\end{proof}

\begin{thm} \label{simples}
A simple absolute weight $\gg$-module is a simple tensor module.
\end{thm}

\begin{proof}
Let $M$ be a simple absolute weight $\gg$-module.  Then $M$ is
integrable by Proposition~\ref{int}, and it also satisfies
Theorem~\ref{th1} \eqref{1four}. 
Fix $0\neq m\in M$ and choose $k$ such
that the commutator subalgebra $\gg'$ of the centralizer of $\gg_k$
annihilates $m$.  One checks immediately that there is a $\mathbb
Z$-grading $\gg=\bigoplus_i\gg^i$ such that 
$\gg^0 \simeq gl(k) \supsetplus \gg'$ (in the cases $\gg=\so(\infty)$ or
$\sp(\infty)$ this semidirect sum is direct) and 
\begin{align*}
\gg^1 & \simeq  V^*_k \otimes V' \textrm{, } \gg^{-1}\simeq V_k\otimes V'_* \textrm{, } 
\gg^i =0\text{ if } |i|>1 &\textrm{for }\gg=\sl(\infty),\\
\gg^{2} & \simeq\Lambda^2(W_k)\textrm{, } 
\gg^{1} \simeq W_k\otimes V' \textrm{, } 
\gg^{-1} \simeq W^*_k\otimes V' \textrm{, } 
\gg^{-2} \simeq \Lambda^2(W^*_k), & \\
\gg^i &=0 \text{ if } |i|>2 &\textrm{for }\gg=\so(\infty),\\
\gg^{2} & \simeq S^2(W_k) \textrm{, } 
\gg^{1} \simeq W_k\otimes V' \textrm{, } 
\gg^{-1}  \simeq W^*_k\otimes V' \textrm{, } 
\gg^{-2} \simeq S^2(W^*_k), & \\
\gg^i & =0 \textrm{ if } |i|>2 &\textrm{for } \gg=\sp(\infty).
\end{align*}
Here $V'$ and $V'_*$ stand respectively for the natural and conatural
module of $\gg' \simeq \gg$, and for $\gg = \so(\infty)$,
$\sp(\infty)$, we take $W_k$ and $W_k^*$ to denote $k$-dimensional isotropic subspaces of $V_k$ such that $V_k = W_k \oplus W_k^*$.

Let $\gg^+=\bigoplus_{i>0}\gg^i$. We claim that the $\gg^+$-invariant part of $M$, denoted $M^{\gg^+}$, is nonzero. 
Note first that, since $\gg^2$ acts locally nilpotently on $M$, without loss of generality we may assume that $\gg^2 \cdot m=0$.
Next observe that $U(\gg^+) \cdot m=S(\gg^1) \cdot m$, and
$S(\gg^1)\simeq S(V'\oplus\dots\oplus V') \simeq S(V')^{\otimes k}$ is isomorphic as a $\gg'$-module to a direct sum of
$\tilde V_{\lambda}$ for some set of $\lambda$ satisfying
$\#\lambda\leq k$ for $\gg=\sl(\infty)$ 
and $\# \lambda \leq 2k$ for $\gg=\so(\infty)$ or $\sp(\infty)$.
By Lemma~\ref{nonzeromarks}, there exists a finite-dimensional subspace $W\subset \gg^1$ such that
$S(W)$ generates $S(\gg^1)$ as a $\gg'$-module. 
If  $\gg=\sl(\infty)$ one can take $W= V^*_k\otimes V'_k$, if
$\gg=\so(\infty)$ or $\sp(\infty)$ one can take $W=W_k\otimes V'_{2k}$.
Since $M$ is
integrable, $W$ acts nilpotently on $M$. Hence $S^{>p }(W) \cdot m=0$ for some
$p$. Then  $S^{>p }(\gg^1) \cdot m=0$, and the latter implies $M^{\gg^+}\neq 0$.

It is easy to see that the irreducibility of $M$ implies the
irreducibility of $M^{\gg^+}$ as a $\gg^0$-module.  
Moreover, the $\gg^0$-module $M^{\gg^+}$ is isomorphic to a
subquotient of $S (\gg^1)$.

If $\gg=\so(\infty)$ or $\sp(\infty)$, then $M^{\gg^+}\simeq W_{\mu}\otimes V'_{\nu}$,
where $W_{\mu}$ is a simple $gl(k)$-module with highest weight $\mu=(\mu_1,\dots,\mu_k)$.
It is easy to check that integrability of $M$ implies that $\mu_k$ is an integer and
$\mu_k\geq\nu_1$. Let
$\lambda :=(\mu_1,\dots,\mu_k,\nu_1,\nu_2,\dots)$. Then
$V_\lambda^{\gg^+}\simeq  M^{\gg^+}$, and this yields  
$M\simeq V_\lambda$ since  the induced module $U(\gg) \otimes_{U(\gg^0  \oplus\gg^+)} M^{\gg^+}$ has a unique simple quotient.

Let $\gg=\sl(\infty)$. Then $M^{\gg^+}$ is the restriction of $W_\mu\otimes V'_{\nu}$
to $\gg^0 \subset gl(k)\oplus gl(\infty)$. The integrability of $M$ implies that each $\mu_i$ is a nonpositive integer. Set
$\lambda :=(\nu,(-\mu_k,\dots,-\mu_1))$. Then we have an isomorphism of $\gg_0$-modules
$V_\lambda^{\gg^+}\simeq  M^{\gg^+}$. 
By the same argument as in the previous case,
we have $M\simeq V_\lambda$.
\end{proof}

\begin{rem}
In \cite{PS} certain categories $\operatorname{Tens}_\gg$ and $\widetilde{\operatorname{Tens}}_\gg$ are introduced and studied in detail.  The simple objects of both $\operatorname{Tens}_\gg$ and $\widetilde{\operatorname{Tens}}_\gg$ are the same as the simple objects of $\Tg$, and in fact these three categories form the following chain:
$$\Tg \subset \operatorname{Tens}_\gg \subset \widetilde{\operatorname{Tens}}_\gg.$$
However, the objects of the categories $\operatorname{Tens}_\gg$ and $\widetilde{\operatorname{Tens}}_\gg$ generally have infinite length.  In the present paper we will not make use of the categories $\operatorname{Tens}_\gg$ and $\widetilde{\operatorname{Tens}}_\gg$, and refer the interested reader to \cite{PS}.
\end{rem}

We now introduce notation $\Gamma^\mathrm{wt}$ for the functor from $\operatorname{Int}_\gg$ to the category of absolute weight modules given by
$$\Gamma^\mathrm{wt} (M) = \bigcap_\hh \Gamma^\mathrm{wt}_\hh (M),$$
where $\hh$ runs over all splitting Cartan subalgebras of $\gg$.

\begin{lemma} \label{rtadjoint}
For any $M \in \operatorname{Int}_\gg$, the module
$\Gamma^\mathrm{wt} (\Gamma_\gg (M^*))$ is injective in the category of absolute weight modules.  Furthermore, any finite length injective module in the category of absolute weight modules is injective in $\Tg$.
\end{lemma}

\begin{proof}
$\Gamma^\mathrm{wt}$ is a right adjoint to the inclusion functor from the category of absolute weight modules to $\operatorname{Int}_\gg$.  To see this, consider that the image of any homomorphism from an absolute weight module to a module $Y \in \operatorname{Int}_\gg$ is automatically contained in $\Gamma^\mathrm{wt} (Y)$.
Since right adjoint functors take injective modules to injective modules, the lemma follows from the fact that $\Gamma_\gg (M^*)$ is injective for any integrable $\gg$-module $M$, which is \cite[Proposition 3.2]{PS}.  The second statement is clear.
\end{proof}

\begin{prop}
For each $r$, the module $T^{\leq r}$ is injective in the category of absolute weight modules and in $\Tg$.
\end{prop}

\begin{proof}
We consider the case $\gg = \sl(\infty)$, and note that the other cases are similar.  We will show that $\Gamma^\mathrm{wt} \big( (T^{q,p})^* \big)$ is injective in the category of absolute weight modules, and furthermore that it has a direct summand isomorphic to $T^{p,q}$.  Since any direct summand of an injective module is itself injective, we see immediately that $T^{p,q}$ is injective in the category of absolute weight modules.  As $T^{p,q}$ has finite length, it is also injective in the category $\Tg$.

It is clear that for each $i > 0$, one has $\operatorname{Hom}_{\gg_i} (N, T^{q,p}) \neq 0$ for only finitely many non-isomorphic simple $\gg_i$-modules $N$.  Thus by \cite[Lemma 4.1]{PS} its algebraic dual $(T^{q,p})^*$ is integrable.  That is $\Gamma_\gg \big((T^{q,p})^*\big) = (T^{q,p})^*$, so Lemma~\ref{rtadjoint} implies that $\Gamma^\mathrm{wt} \big((T^{q,p})^* \big)$ is injective in the category of absolute weight modules.

It remains to show that $T^{p,q}$ is a direct summand of $\Gamma^\mathrm{wt} \big( (T^{q,p})^* \big)$.  One may check that $\Gamma^\mathrm{wt} \big( (T^{0,p})^* \big) = T^{p,0}$ and $\Gamma^\mathrm{wt} \big( (T^{q,0})^* \big) = T^{0,q}$.  Thus it remains to consider the case that $p$ and $q$ are both at least $1$.
 
Fix a splitting Cartan subalgebra $\hh \subset \gg$ compatible with the exhaustion of $\gg$, and let $\{ e_i \mid i \in \ZZ_{>0} \}$ and $\{ e_i^* \mid i \in \ZZ_{>0} \}$ be the dual bases of $V$ and $V_*$ associated to $\hh$ and such that $V_n$ is spanned by $\{ e_1 , \ldots e_n \}$.  The space $(T^{q,p})^*$ is a completion of $T^{p,q}$ in the sense that any element of $(T^{q,p})^*$ can be expressed (uniquely) as a formal sum
\begin{equation} \label{formalsum}
\sum_{i_1, \ldots, i_{p+q}} a_{i_1, \ldots, i_{p+q}} e_{i_1} \otimes \cdots \otimes e_{i_p} \otimes e_{i_{p+1}}^* \otimes \cdots \otimes e_{i_{p+q}}^*.
\end{equation}

Note that $\mathrm{Id} \in (T^{1,1})^*$.  Thus we may define a linear map
\begin{align*}
\iota_{p1} : (T^{q-1,p-1})^* & \rightarrow (T^{q,p})^* \\
e_{i_1} \otimes \cdots \otimes e_{i_{p-1}} \otimes e_{i_p}^* \otimes \cdots \otimes e_{i_{p+q-2}}^* & \mapsto e_{i_1} \otimes \cdots \otimes e_{i_{p-1}} \otimes \mathrm{Id} \otimes e_{i_p}^* \otimes \cdots \otimes e_{i_{p+q-2}}^*.
\intertext{and define analogously}
\iota_{rs} : (T^{q-1,p-1})^* & \rightarrow (T^{q,p})^* \\
\end{align*}
for $1\leq r \leq p$ and $1 \leq s \leq q$.  Then $\iota_{rs} \big( \Gamma^\mathrm{wt} \big((T^{q-1,p-1})^* \big) \big) \subset \Gamma^\mathrm{wt} \big((T^{q,p})^* \big) $ for each $1\leq r \leq p$ and $1 \leq s \leq q$.

We claim that
\begin{equation} \label{decomp}
\Gamma^\mathrm{wt} \big((T^{q,p})^* \big) = T^{p,q} \oplus \sum_{\substack{1\leq r \leq p \\ 1 \leq s \leq q}} \iota_{rs} \big( \Gamma^\mathrm{wt} \big((T^{q-1,p-1})^* \big) \big).
\end{equation}
That the intersection of $T^{p,q}$ with the submodule $$\sum_{\substack{1\leq r \leq p \\ 1 \leq s \leq q}} \iota_{rs} \big( \Gamma^\mathrm{wt} \big((T^{q-1,p-1})^* \big) \big)$$ is trivial follows from the observation that every nonzero element of the latter submodule has infinitely many nonzero coefficients when written in the form \eqref{formalsum}.  Hence to show \eqref{decomp} it remains to show only that the left-hand side is contained in the right-hand side.

Let $\mathcal B (M):=\bigcup_{n}M^{\gg'_n}$, where $\gg'_n \simeq \gg$ is the commutator subalgebra of the centralizer of $\gg_n$.  Recall from \cite{PS} that 
$(T^{q,p})^*$ has finite exhaustive socle filtration. By Theorem \ref{th1} any vector $m \in \soc \big( \Gamma^\mathrm{wt}\big((T^{q,p})^* \big) \big)$ is
annihilated by some $\gg'_n$. It follows by induction on the length of
the socle filtration that any $m\in \Gamma^\mathrm{wt}\big((T^{q,p})^* \big)$
is annihilated by some $\gg'_n$. Hence $\Gamma^\mathrm{wt}\big((T^{q,p})^* \big)\subset \mathcal B\big((T^{q,p})^* \big)$.

From \cite{PStyr} it follows that $T^{q,p}$ has the trivial module as a quotient only if $p = q$, in which case it must be the image of some $p$-fold contraction.  Since $T^{q,p}$ has finite length, we can use Lemma 6.6 from \cite{PS} to pass to the dual and deduce that the $\gg$-invariants of $(T^{q,p})^*$ are nonzero only if $p = q$, in which case they are the span of the $p$-fold tensor powers of $\mathrm{Id} \in (T^{1,1})^*$.  Observe that $V$ considered as a $\gg'_n$-module decomposes into the direct sum of the natural $\gg'_n$-module and the trivial module $V_n$, while $V_*$ has a similar decomposition.  Therefore for each $n$ we have 
\begin{equation}\label{aux1}
\big((T^{q,p})^*\big)^{\gg'_n}=\sum_{j=0}^{\mathrm{min}(p,q)}\sum\iota'_{r_1s_1}\cdots\iota'_{r_js_j}m_j(r_1,s_1,\dots,r_j,s_j),
\end{equation}
where 
$m_j(r_1,s_1,\dots,r_j,s_j)\in V_n^{\otimes p-j}\otimes
(V_n^*)^{\otimes q-j}$ and $\iota'$ are analogs of $\iota$ for
$\mathrm{Id}' := \sum_{i >n} e_i \otimes e_i^*$. But $\iota'_{rs}(m)-\iota_{rs}(m)\in  V_n^{\otimes p-j+1}\otimes
(V_n^*)^{\otimes q-j+1}$ for any
$m \in  V_n^{\otimes p-j}\otimes (V_n^*)^{\otimes q-j}$. Hence we can use $\iota$ instead of
$\iota'$ in the right hand side of (\ref{aux1}). This implies
$$\big((T^{q,p})^*\big)^{\gg'_n}=(T^{p,q})^{\gg'_n}\oplus \sum_{\substack{1\leq r \leq p \\ 1 \leq s \leq q}} \iota_{rs}\big((T^{q-1,p-1})^*\big)^{\gg'_n},$$
and in turn
$$\mathcal B \big((T^{q,p})^*\big)=T^{p,q}\oplus \sum_{\substack{1\leq    r \leq p \\ 1 \leq s \leq q}} \iota_{rs}\mathcal B\big((T^{q-1,p-1})^*\big).$$ 
Hence \eqref{decomp} holds.
\end{proof}

\begin{cor}\label{newlab}
\begin{enumerate}
\item $\tilde V_\lambda$ is injective in $\Tg$.
\item $\tilde V_\lambda$ is an injective hull of $V_\lambda$ in $\Tg$.
\item Every indecomposable injective module in $\Tg$ is isomorphic to $\tilde V_\lambda$ for some $\lambda$.
\item \label{cor4} Every module $M \in \Tg$ is isomorphic to a submodule of the direct sum of finitely many copies of $T^{\leq r}$ for some $r$.
\item A $\gg$-module $M$ is a tensor module if and only if $M \in \Tg$.
\end{enumerate}
\end{cor}

\begin{proof}
\begin{enumerate}
\item Each module $\tilde V_\lambda$ is a direct summand of $T^{\leq r}$ for some $r$, and a direct summand of an injective module is injective.
\item Any indecomposable injective module is an injective hull of its socle, and $\soc(\tilde{V}_\lambda) = V_\lambda$ by \eqref{nabla}.
\item Every indecomposable injective module in $\Tg$ has a simple socle, which must be isomorphic to $V_\lambda$ for some $\lambda$ by Theorem~\ref{simples}.
\item Let $M \in \Tg$.  Then $\soc(M)$ admits an injective homomorphism into a direct sum of finitely many copies of $T^{\leq r}$ for some $r$.  Since the latter is injective in $\Tg$, this homomorphism factors through the inclusion $\soc(M) \hookrightarrow M$.  The resulting homomorphism must be injective because its kernel has trivial intersection with $\soc(M)$.
\item A tensor module is by definition a subquotient of a direct sum of finitely many copies of $T^{\leq r}$ for some $r$, hence it is clearly finite length.  Furthermore, any subquotient of an absolute weight module must be an absolute weight module, so any tensor module must be in $\Tg$.  The converse was seen in \eqref{cor4}. 
\end{enumerate}
\end{proof}

\section{Koszulity of $\Tg$}

For $r \in \ZZ_{\geq 0}$, let $\Tg^r$ be the full abelian subcategory of $\Tg$ whose simple
objects are submodules of $T^{\leq r}$. 
Then $\Tg=\limarr \Tg^r$. Moreover, $T^{\leq r}$ is an injective
cogenerator of $\Tg^r$. 
Consider the finite-dimensional algebra
$\mathcal{A}_\gg^r:=\operatorname{End}_\gg T^{\leq r}$ and 
the direct limit algebra $\mathcal{A}_\gg=\limarr \mathcal{A}_\gg^{r} = \operatorname{End}_\gg T$.

Let $\mathcal{A}_\gg^r$-$\mathrm{mof}$ denote the category of unitary finite-dimensional $\mathcal{A}_\gg^r$-modules, and
$\mathcal{A}_\gg$-$\mathrm{mof}$ the category of locally unitary finite-dimensional $\mathcal{A}_\gg$-modules.

\begin{prop}\label{equiv} 
The functors
  $\operatorname{Hom}_\gg(\, \cdot \,, T^{\leq r})$ and
  $\operatorname{Hom}_{\mathcal A_\gg^r}(\, \cdot \,, T^{\leq r})$ are mutually inverse
  antiequivalences of the categories  $\Tg^r$ and $\mathcal{A}_\gg^r$-$\mathrm{mof}$. 
\end{prop}
\begin{proof}
Consider the opposite category $(\Tg^r)^\mathrm{opp}$.  It has finitely many simple objects and enough projectives, and any object has finite length.  Moreover, $T^{\leq r}$ is a projective generator of $(\Tg^r)^\mathrm{opp}$.  By a well-known result of Gabriel \cite{G}, 
the functor
$$\operatorname{Hom}_{(\Tg^r)^\mathrm{opp}} (T^{\leq r}, \, \cdot \,) = \operatorname{Hom}_\gg (\, \cdot \,, T^{\leq r}) : (\Tg^r)^\mathrm{opp} \rightarrow \mathcal{A}_\gg^r \textrm{-mof}$$
is an equivalence of categories. 

We claim that $\operatorname{Hom}_{\mathcal{A}_\gg^r} (\, \cdot \,, T^{\leq r})$ is an inverse to $\operatorname{Hom}_\gg (\, \cdot \,, T^{\leq r})$.  For this it suffices to check that $\operatorname{Hom}_{(\Tg^r)^\mathrm{opp}} (T^{\leq r}, \, \cdot \,)$ is a right adjoint to $\operatorname{Hom}_{\mathcal{A}_\gg^r} (\, \cdot \,, T^{\leq r})$, i.e.\ that 
\begin{equation*} 
\operatorname{Hom}_{\mathcal{A}_\gg^r} (X, \operatorname{Hom}_{(\Tg^r)^\mathrm{opp}} (T^{\leq r}, M)) \simeq \operatorname{Hom}_{(\Tg^r)^\mathrm{opp}} (\operatorname{Hom}_{\mathcal{A}_\gg^r} (X, T^{\leq r}),M)
\end{equation*}
for any $X \in \mathcal{A}_\gg^r \textrm{-mof}$ and any $M \in \Tg^r$.  We have 
\begin{align*}
\operatorname{Hom}_{\mathcal{A}_\gg^r} (X, \operatorname{Hom}_{(\Tg^r)^\mathrm{opp}} (T^{\leq r}, M))
& = \operatorname{Hom}_{\mathcal{A}_\gg^r} (X, \operatorname{Hom}_\gg (M, T^{\leq r})) \\
& \stackrel{\Psi}{\simeq}  \operatorname{Hom}_{\mathcal{A}_\gg^r \otimes U (\gg) } (X \otimes M, T^{\leq r}) \\
& = \operatorname{Hom}_{U (\gg) \otimes \mathcal{A}_\gg^r } (M \otimes X, T^{\leq r}) \\
&  \stackrel{\Theta}{\simeq} \operatorname{Hom}_\gg (M , \operatorname{Hom}_{\mathcal{A}_\gg^r} (X, T^{\leq r}))\\
& = \operatorname{Hom}_{(\Tg^r)^\mathrm{opp}} (\operatorname{Hom}_{\mathcal{A}_\gg^r} (X, T^{\leq r}),M),
\end{align*}
where $\Psi(\varphi) (x \otimes m) = \varphi (x) (m)$ and $(\Theta (x) (m))(\psi) = \psi(m \otimes x)$ for $x \in X$, $m \in M$, $\varphi \in  \operatorname{Hom}_{\mathcal{A}_\gg^r} (X, \operatorname{Hom}_\gg (M, T^{\leq r}))$, and $\psi \in  \operatorname{Hom}_{U (\gg) \otimes \mathcal{A}_\gg^r } (M \otimes X, T^{\leq r})$.
\end{proof}

In order to relate the category $\mathcal{A}_\gg$-mof with the categories $\mathcal{A}_\gg^r$-mof for all $r \geq 0$, we need to establish some basic facts about the algebra $\mathcal{A}_\gg$.  Note first that by \cite{PStyr} 
$\operatorname{Hom}_{\sl(\infty)}(T^{p,q},T^{r,s})=0$ unless
$p-r=q-s\in \mathbb Z_{\geq 0}$, and for $\gg = \so(\infty)$, $\sp(\infty)$, $\operatorname{Hom}_\gg(T^{p},T^{q})=0$ unless
$p-q\in 2\mathbb Z_{\geq 0}$.
Furthermore, put 
\begin{align*}
(\mathcal A_\gg)^{p,q}_i&=\operatorname{Hom}_\gg(T^{p,q}, T^{p-i,q-i})
&\textrm{for } \gg=\sl(\infty) \\
\intertext{and}
(\mathcal A_\gg)_i^p&=\operatorname{Hom}_\gg(T^p,T^{p-2i})
&\textrm{for }\gg=\so(\infty) \textrm{, } \sp(\infty).\\ 
\intertext{
Then one can define a $\mathbb Z_{\geq 0}$-grading on $\mathcal A^r_\gg$ by
setting}
(\mathcal A^r_\gg)_i &=\bigoplus_{p+q\leq r} (\mathcal A_\gg)^{p,q}_i &\textrm{for } \gg=\sl(\infty) \\
\intertext{and}
(\mathcal A^r_\gg)_i &=\bigoplus_{p\leq r}  (\mathcal A_\gg)_i^p
&\textrm{for }\gg=\so(\infty) \textrm{, } \sp(\infty).\\ 
\intertext{
It also follows from the results of  \cite{PStyr} that} 
(\mathcal A^r_\gg)_0 & =\bigoplus_{p+q\leq r} \operatorname{End}_\gg(T^{p,q})=\bigoplus_{p+q\leq r}\mathbb C[S_p\times S_q] &\textrm{for } \gg=\sl(\infty) \\
\intertext{and}
(\mathcal A^r_\gg)_0 & =\bigoplus_{p\leq r} \operatorname{End}_\gg(T^{p})=\bigoplus_{p\leq r} \mathbb C[S_p] &\textrm{for }\gg=\so(\infty) \textrm{, } \sp(\infty).
\end{align*}
Hence $(\mathcal A^r_\gg)_0$ is semisimple.

In addition, we have
\begin{align*}
(\mathcal A_\gg)^{p,q}_i (\mathcal A_\gg)^{r,s}_j &=0 \textrm{ unless }p =r -j \textrm{, }q =s-j  &\textrm{for } \gg=\sl(\infty) \\
\intertext{and}
(\mathcal A_\gg)^p_i (\mathcal A_\gg)^r_j &=0 \textrm{ unless }p =r -2j  
&\textrm{for }\gg=\so(\infty) \textrm{, } \sp(\infty).\\ 
\intertext{
This shows that for each $r$,} 
\bar{\mathcal A^r_\gg} &:=\bigoplus_{p+q > r} \bigoplus_{i \geq 0}  (\mathcal A_\gg)_i^{p,q}
&\textrm{for } \gg=\sl(\infty) \\
\intertext{or}
\bar{\mathcal A^r_\gg} &:=\bigoplus_{p > r} \bigoplus_{i \geq 0}  (\mathcal A_\gg)_i^p
&\textrm{for }\gg=\so(\infty) \textrm{, } \sp(\infty)
\end{align*}
is a $\ZZ_{\geq 0}$-graded ideal in $\mathcal A_\gg$ such that $\mathcal A^r_\gg \oplus \bar{\mathcal A^r_\gg} = \mathcal A_\gg$.
Hence each unitary $\mathcal A^r_\gg$-module $X$ admits a canonical $\mathcal A_\gg$-module structure with $\bar{\mathcal A^r_\gg} X = 0$, and thus becomes a locally unitary $\mathcal A_\gg$-module.  This allows us to claim simply that  
$$\mathcal{A}_\gg\textrm{-mof}=\limarr  \mathcal{A}_\gg^r \textrm{-mof}.$$

Moreover, Proposition \ref{equiv} now implies the following.

\begin{cor}\label{bigequiv} 
The functors
  $\operatorname{Hom}_\gg(\, \cdot \,, T)$ and
  $\operatorname{Hom}_{\mathcal A_\gg}(\, \cdot \,, T)$ are mutually inverse
  antiequivalences of the categories  $\Tg$ and
  $\mathcal{A}_\gg$-$\mathrm{mof}$.
\end{cor}

We now need to recall the definition of a Koszul ring. See
\cite{BGS}, where this notion is studied extensively, and, in
particular, several equivalent definitions are given. According to
 Proposition 2.1.3 in \cite{BGS}, a $\mathbb Z_{\geq 0}$-graded ring $A$
is \emph{Koszul} if $A_0$ is a semisimple ring and for any two graded $A$-modules
$M$ and $N$  of pure weight $m, n \in \mathbb{Z}$ respectively,
$\operatorname{ext}_A^i(M,N)=0$ unless $i=m-n$, where
$\operatorname{ext}_A^i$ denotes the ext-group in the category of
$\mathbb{Z}$-graded $A$-modules. 

In the rest of this section we show that 
$\mathcal A^r_\gg$ is a Koszul ring. 

We start by introducing the following notation: for any partition $\mu$, we set
\begin{align*}
\mu^+ & := \{ \textrm{partitions } \mu' \mid |\mu'| = |\mu| + 1 \textrm{ and } \mu'_i \neq \mu_i \textrm{ for exactly one }i \}, \\
\mu^- & := \{  \textrm{partitions } \mu' \mid |\mu'| = |\mu| - 1 \textrm{ and } \mu'_i \neq \mu_i \textrm{ for exactly one }i \}.
\end{align*} 
For any pair of partitions $\lambda = (\lambda^1,\lambda^2)$, we define
\begin{align*}
\lambda^+ & := \{ \textrm{pairs of partitions } \eta \mid \eta^1\in{\lambda^1}^+, \eta^2=\lambda^2
 \}, \\
\lambda^- & := \{  \textrm{pairs of partitions } \eta \mid \eta^1=\lambda^1, \eta^2\in{\lambda^2}^-  \}.
\end{align*} 

\begin{lemma} \label{les}
For any simple object $V_\lambda$ of $\Tg$, there is an exact sequence
$$0\to V^+_\lambda\to V\otimes V_\lambda\to V_\lambda^-\to 0,$$
where
\begin{align*}
V^+_\lambda&=\bigoplus_{\eta\in\lambda^+}V_\eta\\
 V^-_\lambda&=\bigoplus_{\eta\in\lambda^-}V_\eta.
\end{align*}
Moreover, $V^+_\lambda=\operatorname{soc} (V\otimes V_\lambda)$.
\end{lemma}
\begin{proof} We will prove the statement for $\gg=\sl(\infty)$. The other
cases are similar. The fact that the semisimplification of $V\otimes V_\lambda$ is isomorphic to $ V^+_\lambda \oplus V^-_\lambda$ follows from the classical Pieri rule. 

To prove the equality $V^+_\lambda=\operatorname{soc} (V\otimes V_\lambda)$, observe that
$$\operatorname{soc} (V\otimes V_\lambda)= \operatorname{soc}
(V\otimes \tilde{V}_\lambda)=  \operatorname{soc}(T^{|\lambda^1|+1,|\lambda^2|})\cap(V\otimes \tilde{V}_\lambda).$$
The fact that $V^+_\lambda= \operatorname{soc}(T^{|\lambda^1|+1,|\lambda^2|})\cap(V\otimes \tilde{V}_\lambda)$
follows directly from \cite[Theorem 2.3]{PStyr}.

It remains to show that the quotient $(V\otimes V_\lambda) / V^+_\lambda$ is semisimple. This follows again from
\cite[Theorem 2.3]{PStyr}, since all simple subquotients of $(V\otimes V_\lambda) / V^+_\lambda$, i.e.\ all direct summands of
$V^-_\lambda$, lie in $\operatorname{soc}^1(T^{|\lambda^1|+1,|\lambda^2|})$.
\end{proof}

\begin{prop}\label{ext} 
If $\operatorname{Ext}^i_{\Tg}(V_\lambda,V_\mu)\neq  0$, then
\begin{align*}
|\mu^1|-|\lambda^1|=|\mu^2|-|\lambda^2|&=i & \textrm{for } \gg=\sl(\infty)\\
\intertext{and}
|\mu|-|\lambda|&=2i & \textrm{for }  \gg=\so(\infty) \textrm{, } \sp(\infty).
\end{align*}
\end{prop}

\begin{proof} Let $\gg=\sl(\infty)$. We will prove the statement by
  induction on $|\mu|$. The base of induction $\mu=(0,0)$
  follows immediately from the fact that
  $V_{(0,0)}=\mathbb C$ is injective. We assume $\operatorname{Ext}^i_{\Tg}(V_\lambda,V_\mu)\neq  0$.  Without loss of generality we
  may assume that $|\mu^1|>0$. Then there exists a pair of partitions $\eta$ such that $\mu \in \eta^+$.
   Since $V_\mu$ is a direct summand of $V^+_\eta$, we have $\operatorname{Ext}^i_{\Tg}(V_\lambda, V^+_\eta)\neq  0$.

Consider the short exact sequence from Lemma~\ref{les}
$$0\to V_\eta^+\to V \otimes V_\eta \to V_\eta^-\to 0.$$
The associated long exact sequence implies that either $\operatorname{Ext}^i_{\Tg}(V_\lambda, V \otimes V_\eta)\neq  0$ or $\operatorname{Ext}^{i-1}_{\Tg}(V_\lambda, V^-_\eta)\neq  0$.  In the latter case, the inductive hypothesis implies that
$$| \eta^1| - | \lambda_1 | = (| \eta^2| -1) - | \lambda^2| = i-1.$$ 
The condition in the statement of the proposition follows, as $| \eta^1| = |\mu^1| -1$ and $ |\eta^2 | = | \mu^2 |$. 

Now assume that $\operatorname{Ext}^i_{\Tg}(V_\lambda, V \otimes V_\eta)\neq  0$.  Let 
$$0\to V_\eta\to M_0\to M_1\to\dots$$
be a minimal injective resolution of $V_\eta$ in $\Tg$.
By the inductive hypothesis, $\operatorname{Ext}^j_{\Tg}(V_\nu,V_\eta)\neq 0$
implies 
\begin{equation} \label{etanu}
|\eta^1|-|\nu^1|=|\eta^2|-|\nu^2|=j.
\end{equation}
By the minimality of the resolution, it has no nontrivial direct sum decomposition, 
hence $\tilde V_\nu$ appears as a direct summand of $M_j$ only
if \eqref{etanu} holds. That is, $M_j=\oplus \tilde V_\nu$ for some set of $\nu$ such that $|\nu^1| = |\eta^1|-j$ and $|\nu^2| = |\eta^2|-j$.   

Furthermore, 
$$0\to V \otimes V_\eta\to V \otimes M_0\to V \otimes M_1\to\dots$$
is an injective resolution of  $V \otimes V_\eta$.  Thus $\operatorname{Hom}_\gg ( V_\lambda , V \otimes M_i) \neq 0$ implies $|\lambda^1| = |\eta^1| - i + 1$ and $|\lambda^2| = |\eta^2| - i$, and the proof for $\gg = \sl(\infty)$ is complete.

The proof for $\gg=\so(\infty)$, $\sp(\infty)$ is similar, and
we leave it to the reader.
\end{proof}

Recall that any $\gg$-module $W$ has a well-defined socle filtration $$0 \subset \soc^0 (W) = \soc (W) \subset \soc^1 (W) \subset \cdots$$
where $\soc^i(W) := \pi^{-1}_{i-1} (\soc (W / \soc^{i-1}(W))$ and $\pi_{i-1} : W \rightarrow W/\soc^{i-1}(W)$ is the projection.  Similarly, any $\mathcal A_\gg$-module $X$ has a radical filtration
$$ \cdots \subset \rad^1(X) \subset \rad^0(X) = \rad (X) \subset X$$
where $\rad(X)$ is the joint kernel of all surjective $\mathcal A_\gg$-homomorphisms $X \rightarrow X'$ with $X'$ simple, and $\rad^i(X) = \rad (\rad^{i-1}(X))$.

Note furthermore that the $\operatorname{Ext}$'s in the category $\Tg$ differ essentially from the $\operatorname{Ext}$'s in $\gg$-mod.  In particular, as shown in \cite{PS}, $\operatorname{Ext}_\gg^1 (V_\lambda, V_\mu)$ is uncountable dimensional whenever nonzero, whereas $\operatorname{Ext}_{\Tg}^1 (V_\lambda, V_\mu)$ is always finite dimensional by Corollary~\ref{bigequiv}.  Here are two characteristic examples.
\begin{enumerate}
\item Consider the exact sequence of $\gg$-modules
$$ 0 \rightarrow V \rightarrow (V_*)^* \rightarrow (V_*)^* / V \rightarrow 0.$$
The $\gg$-module $(V_*)^* / V$ is trivial, and any vector in $\operatorname{Ext}_{\sl(\infty)}^1 (\CC , V)$ determines a unique $1$-dimensional subspace in $(V_*)^* / V$.  On the other hand, $\operatorname{Ext}_{\T_{\sl(\infty)}}^1 (\CC , V) = 0$ by Proposition~\ref{ext}.
\item Each nonzero vector of $\operatorname{Ext}_{\sl(\infty)}^1 (\CC , \sl(\infty))$ corresponds to a $1$-dimensional trivial quotient of $\soc^1 ((\sl(\infty)_*)^*)$ (see \cite{PS}).  The nonzero vectors of the $1$-dimensional space $\operatorname{Ext}_{\T_{\sl(\infty)}}^1 (\CC , \sl(\infty))$ on the other hand correspond to the unique $1$-dimensional quotient of $\soc^1 ((\sl(\infty)_*)^*)$ which determines an absolute weight module, namely $\tilde{\sl(\infty)} / \sl(\infty) = (V \otimes V_*) / \sl(\infty)$.
\end{enumerate}

The following is the main result of this section.

\begin{thm}\label{koszulity} 
The ring $\mathcal A^r_\gg$ is Koszul.
\end{thm}

\begin{proof} 
According to \cite[Proposition 2.1.3]{BGS}, it suffices to prove that  unless $i = m - n$, one has $\operatorname{ext}^i_{\mathcal{A}^r_\gg} (M , N) = 0$ for any pure $\mathcal{A}^r_\gg$-modules $M$, $N$ of weights $m$, $n$ respectively.  We will prove that unless $i = m - n$, one has $\operatorname{ext}^i_{\mathcal{A}_\gg} (M , N) = 0$ for any simple pure $\mathcal{A}_\gg$-modules $M$, $N$ of weights $m$, $n$ respectively.  Since any $\mathcal A^r_\gg$-module admits a canonical $\mathcal A_\gg$-module structure, it follows that $\operatorname{ext}^i_{\mathcal{A}^r_\gg} (M , N) = 0$ for any simple pure $\mathcal{A}^r_\gg$-modules $M$, $N$ of weights $m$, $n$ respectively unless $i = m-n$.  The analogous statement for arbitrary $\mathcal A^r_\gg$-modules of pure degree follows, since all such modules are semisimple. 

Let $X_\lambda$ (respectively, $\tilde X_\lambda$) be the simple $\mathcal A_\gg$-module which is the image of $V_\lambda$ (resp., $\tilde V_\lambda$) under the antiequivalence of Corollary~\ref{bigequiv}. Then  $\tilde X_\lambda$
is a projective cover of  $X_\lambda$. Proposition~\ref{ext} implies
that  $\operatorname{Ext}^i_{\mathcal A_\gg}(X_\mu,X_\lambda)=0$ unless
$|\mu^1|-|\lambda^1|=|\mu^2|-|\lambda^2|=i$. We consider a minimal
projective resolution of $X_\mu$
\begin{equation}\label{res}
\dots\to P^1\to P^0\to 0
\end{equation}
and claim that it must have the property $P^i \simeq \oplus\tilde X_\nu$ for some set of $\nu$ with $|\mu^1|-|\nu^1|=|\mu^2|-|\nu^2|=i$.

On the other hand, by \cite{PStyr} if $V_\nu$ is a simple constituent of
$\operatorname{soc}^i (\tilde V_\mu)$, or if under the antiequivalence $X_\nu$ is a simple constituent of
$\operatorname{rad}^i \tilde X_\mu$,  then   $|\mu^1|-|\nu^1|=|\mu^2|-|\nu^2|=i$.
Therefore we see that in the above resolution the image of $\rad^j(P^i)$ lies in $\rad^{j+1}(P^{i-1})$. Now it is clear that we can endow the resolution \eqref{res} with a
$\mathbb Z$-grading by setting the degree of $X_\mu$ to be an
arbitrary integer $n$.  Indeed, one should take each simple constituent of $P^i$ as an $(\mathcal A_\gg)_0$-module which lies in $\rad^j (P_i)$ and not in $\rad^{j+1} (P_i)$ to have degree $n+ i + j +1$.
This immediately implies that $\operatorname{ext}_{\mathcal A_\gg}^i(X_\mu, X_\lambda)=0$ unless the difference between the weights of $X_\lambda$ and $X_\mu$ is $i$.
\end{proof}

\section{On the structure of $\mathcal A_\gg$}

According to \cite{BGS} the Koszulity of   $\mathcal A^r_\gg$ for all
$r$ implies that  $\mathcal A^r_\gg$ is a quadratic algebra generated by  $(\mathcal A^r_\gg)_0$ and  $(\mathcal A^r_\gg)_1$, i.e.\  $\mathcal A^r_\gg\simeq T_{(\mathcal A^r_\gg)_0}( (\mathcal A^r_\gg)_1)/(R^r)$,
where $(R^r)$ is the two-sided ideal generated by some  $(\mathcal A^r_\gg)_0$-bimodule $R^r$ in $(\mathcal A^r_\gg)_1\otimes _ {(\mathcal A^r_\gg)_0} (\mathcal A^r_\gg)_1$.  Moreover, it is easy to see that $\mathcal A_\gg$ is isomorphic to the quotient $T_{(\mathcal A_\gg)_0}( (\mathcal A_\gg)_1)/(R)$, where $R = \limarr R^r$.
In this section we describe $(\mathcal A_\gg)_1$ and $R$.

In what follows we fix inclusions $S_n \subset S_{n+1}$ such that $S_{n+1}$ acts on the set $\{1,2, \ldots, n+1 \}$ and $S_n$ is the stabilizer of $n+1$. We start with the following lemma.
\begin{lemma} \label{contractions} If $\gg=\sl(\infty)$, then
$\operatorname{Hom}_\gg(T^{p,q}, T^{p-1,q-1})$ as a
left module over $\mathbb C[S_{p-1}\times S_{q-1}]$ is generated by the contractions
\begin{align*} \phi_{i,j} : T^{p,q} & \rightarrow T^{p-1,q-1}, \\
v_1\otimes\cdots\otimes v_p\otimes w_1\otimes\cdots\otimes w_q & \mapsto
\langle v_i,w_j\rangle (v_1\otimes\cdots\hat v_i\cdots\otimes
  v_p\otimes w_1\otimes\cdots\hat w_j\cdots\otimes w_q).
\end{align*}

If $\gg=\so(\infty)$ or $\sp(\infty)$, then
$\operatorname{Hom}_\gg(T^{p},T^{p-2})$ as a
left module over $\mathbb C[S_{p-2}]$ is generated by the
contractions  
\begin{align*} \psi_{i,j} : T^p & \rightarrow T^{p-2},\\
v_1\otimes\cdots\otimes v_p & \mapsto
\langle v_i,v_j\rangle (v_1\otimes\cdots\hat v_i\cdots\otimes\cdots\hat v_j\cdots\otimes v_p),
\end{align*}
where 
$\langle\cdot \, , \, \cdot\rangle $ stands for the symmetric bilinear form on $V$ for  $\gg=\so(\infty)$, and the symplectic bilinear form on $V$ for  $\gg=\sp(\infty)$.
\end{lemma}
\begin{proof} Let  $\gg=\sl(\infty)$ and $\varphi\in \operatorname{Hom}_\gg(T^{p,q}, T^{p-1,q-1})$.  
Theorem 3.2  in \cite{PStyr} claims that  $\soc(T^{p,q})=\cap_{i\leq p,j\leq q}\operatorname{ker}\phi_{i,j}$; moreover, the same result implies that $\soc(T^{p,q})\subset\operatorname{ker}\varphi$. 
Define $$\Phi:T^{p,q}\to \bigoplus_{i\leq p,j\leq q} T^{p-1,q-1}$$ as the
direct sum $\bigoplus_{i,j} \phi_{i,j}$. Then there exists $\alpha:\bigoplus_{i\leq  p,j\leq q} T^{p-1,q-1}\to  T^{p-1,q-1}$ 
such that $\varphi=\alpha\circ\Phi$. But $\alpha=\bigoplus_{i,j} \alpha_{i,j}$
for some $\alpha_{i,j}\in \mathbb C[S_{p-1}\times S_{q-1}]$. Therefore
$\varphi=\sum_{i,j} \alpha_{i,j}\phi_{i,j}$.  This proves the lemma for $\gg = \sl(\infty)$.

We leave the proof in the cases  $\gg=\so(\infty)$, $\sp(\infty)$  to the reader.
\end{proof}

Let $\gg=\sl(\infty)$. Recall that $(\mathcal A_\gg)^{p,q}_i=\operatorname{Hom}_\gg(T^{p,q}, T^{p-i,q-i})$ and that $(\mathcal A_\gg)^{p,q}_0=\mathbb C[S_p\times S_q]$.

\begin{lemma}\label{structure} Let  $\gg=\sl(\infty)$. 
\begin{enumerate}
\item \label{firstly} $(\mathcal  A_\gg)^{p,q}_1$ is isomorphic
  to $\mathbb C[S_p\times S_q]$ as a right $(\mathcal  A_\gg)^{p,q}_0$-module, and the structure of
 a left $(\mathcal A_\gg)^{p-1,q-1}_0$-module is given by
  left multiplication  via the fixed inclusion  
$$(\mathcal A_\gg)^{p-1,q-1}_0=\mathbb C[S_{p-1}\times S_{q-1}]\subset \mathbb C[S_p\times S_q]=(\mathcal  A_\gg)^{p,q}_0.$$
\item \label{secondly} We have
$$(\mathcal A_\gg)_1\otimes_{(\mathcal A_\gg)_0}(\mathcal A_\gg)_1=
\bigoplus_{p,q} ((\mathcal A_\gg)^{p-1,q-1}_1\otimes_{(\mathcal A_\gg)^{p-1,q-1}_0}(\mathcal A_\gg)^{p,q}_1),$$
where $(\mathcal A_\gg)^{p-1,q-1}_1\otimes_{(\mathcal A_\gg)^{p-1,q-1}_0}(\mathcal A_\gg)^{p,q}_1$
is isomorphic to $ \mathbb C[S_p\times S_q]$. Moreover, $(\mathcal A_\gg)^{p-1,q-1}_1\otimes_{(\mathcal A_\gg)^{p-1,q-1}_0}(\mathcal A_\gg)^{p,q}_1$ is a $ (\mathbb C[S_{p-2}\times S_{q-2}], \mathbb C[S_p\times S_q])$-bimodule via the embeddings $\CC[S_{p-2}\times S_{q-2}]\subset \CC[S_{p-1}\times S_{q-1}]\subset \CC[S_p\times S_q]$.
\end{enumerate}
\end{lemma}
\begin{proof}
It is clear that all contractions $\phi_{i,j} \in (\mathcal A_\gg)^{p,q}_1$ can be obtained from $\phi_{p,q}$ via the right $\mathbb C[S_p\times S_q]$-module structure of $(\mathcal A_\gg)^{p,q}_1$. Thus by Lemma \ref{contractions}, as a $\mathbb C[S_p\times S_q]$-bimodule,  $(\mathcal A_\gg)^{p,q}_1$ is generated by the single contraction
$\phi_{p,q}$. Furthermore, the computation 
\begin{align*}
&\sum_{\sigma \in S_p\times S_q}a_\sigma \phi_{p,q}\circ \sigma(v_1\otimes\cdots\otimes v_p\otimes w_1\otimes\cdots\otimes w_q)\\
& = \sum_{\substack{\sigma = (\sigma_1,\sigma_2) \\ \hspace{1ex} \in S_p\times S_q}}a_\sigma \langle v_{\sigma_1(p)},w_{\sigma_2(q)}\rangle
  (v_{\sigma_1(1)}\otimes\cdots\otimes v_{\sigma_1(p-1)}\otimes w_{\sigma_2(1)}\otimes\cdots\otimes w_{\sigma_2(q-1)})=0
\end{align*}
for all $v_1\otimes\cdots\otimes v_p\otimes w_1\otimes\cdots\otimes w_q\in T^{p,q}$
implies $a_\sigma=0$ for all $\sigma \in S_p\times S_q$.
Hence $(\mathcal A_\gg)^{p,q}_1$  is a free right $\mathbb C[S_p\times  S_q]$-module of rank $1$.
On the other hand, for any $\sigma \in S_{p-1}\times S_{q-1}$ we have
$$\sigma \circ \phi_{p,q}= \phi_{p,q}\circ \sigma.$$
This implies part \eqref{firstly}. Part \eqref{secondly} is a direct corollary of part \eqref{firstly}.
\end{proof}

\begin{lemma}  \label{bardef}
Let  $\gg=\sl(\infty)$.  Let $S \simeq S_2\times S_2$ denote the subgroup of $S_p \times S_q$ generated by $(p,p-1)_l$ and
$(q,q-1)_r$, where $(i,j)_l$ and $(i,j)_r$ stand for the transpositions in $S_p$ and $S_q$, respectively. 
Then $R=\bigoplus_{p,q} R^{p,q}$, where 
$$ R^{p,q}=(triv\boxtimes sgn \, \oplus \, sgn\boxtimes triv)\otimes_{\mathbb C[S]}\mathbb C[S_p\times  S_q],$$
and $triv$ and $sgn$ denote respectively the trivial and sign representations of $S_2$.
\end{lemma}
\begin{proof} The statement is equivalent to the equality
$$ R^{p,q}=(1-(p,p-1)_l)(1+(q,q-1)_r)\mathbb C[S_p\times  S_q]\oplus (1+(p,p-1)_l)(1-(q,q-1)_r)\mathbb C[S_p\times  S_q],$$
We have the obvious relations in $\mathcal A_{\sl(\infty)}$
\begin{align*}
\phi_{p-1,q-1}\phi_{p,q}&=\phi_{p-1,q-1}\phi_{p,q}(p,p-1)_l(q,q-1)_r, \\
\phi_{p-1,q-1}\phi_{p,q}(p,p-1)_l&=\phi_{p-1,q-1}\phi_{p,q}(q,q-1)_r.
\end{align*} 
Therefore $1-(p,p-1)_l(q,q-1)_r, (p,p-1)_l-(q,q-1)_r\in R^{p,q}$.
On the other hand, $(1+(p,p-1)_l)(1+(q,q-1)_r)$ and
$(1-(p,p-1)_l)(1-(q,q-1)_r)$ obviously do not belong to $R^{p,q}$. The latter two group algebra elements
generate a right $\mathbb C[S_p\times S_q]$-submodule $\bar R^{p,q} \subset (\mathcal A_\gg)^{p-1,q-1}_1\otimes_{(\mathcal
  A_\gg)^{p-1,q-1}_0}(\mathcal A_\gg)^{p,q}_1$,
and we have 
$$(\mathcal A_\gg)^{p-1,q-1}_1\otimes_{(\mathcal
  A_\gg)^{p-1,q-1}_0}(\mathcal A_\gg)^{p,q}_1=R^{p,q}\oplus \bar R^{p,q}.$$
Hence the statement.
\end{proof}

\begin{cor}\label{selfduality}  
Let  $\gg=\sl(\infty)$. Then $\mathcal A^r_\gg$ is Koszul self-dual, i.e.\  $\mathcal A^r_\gg\simeq  ((\mathcal A^r_\gg)^!)^\mathrm{opp}$.  Furthermore, $\mathcal A_\gg\simeq  (\mathcal A_\gg^!)^\mathrm{opp}$, where $\mathcal A_\gg^! := \limarr (\mathcal A^r_\gg)^!$.
\end{cor}

\begin{proof} 
By definition,  $(\mathcal A^r_\gg)^!=T_{(\mathcal   A^r_\gg)_0}((\mathcal  A^r_\gg)_1^*)/(R^{r \perp})$, 
where  $(\mathcal  A^r_\gg)_1^* = \operatorname{Hom}_{(\mathcal  A^r_\gg)_0} ((\mathcal  A^r_\gg)_1, (\mathcal  A^r_\gg)_0)$,
\cite{BGS}. Note that $(\mathcal (A_\gg)_1^{p,q})^*$ is a 
$(\mathcal(A_\gg)_0^{p,q},\mathcal(A_\gg)_0^{p-1,q-1})$-bimodule. Moreover,
Lemma \ref{structure} (1) implies an isomorphism of
bimodules
$$(\mathcal (A_\gg)_1^{p,q})^*\simeq\mathbb C[S_p\times S_q].$$
Hence  we have an isomorphism of $((\mathcal A^r_\gg)^!)^\mathrm{opp}_0$-bimodules 
$$((\mathcal A^r_\gg)^!)^\mathrm{opp}_1\simeq (\mathcal A^r_\gg)_1.$$
  
One can check that
  $R^\perp=\bar R$, where $\bar R:=\oplus \bar R^{p,q}$, and the
modules $\bar R^{p,q}$ were introduced in the proof of
Lemma~\ref{bardef}. Therefore  
$((\mathcal  A^r)_\gg^!)^{\mathrm {opp}} \simeq T_{(\mathcal  A^r_\gg)_0}((\mathcal  A^r_\gg)_1)/(\bar R^r)$.
Now consider the automorphism $\sigma$ of $\mathbb C[S_p\times S_q]$
defined for all $p$ and $q$ by $\sigma(s, t)=sgn(t)(s, t)$ for
all $s\in S_p$, $t\in S_q$. Recall that $(\mathcal A_\gg)_0=\bigoplus_{p,q}\mathbb C[S_p\times S_q]$.  
Extend $\sigma$ to an automorphism of $T_{(\mathcal  A_\gg)_0}((\mathcal  A_\gg)_1)$ 
by setting $\sigma(x)=x$ for any $x\in (\mathcal A_\gg)_1$. One
immediately observes that $\sigma(R^{p,q})=\bar R^{p,q}$, hence
$\sigma$ induces an isomorphism $\mathcal{A}^r_\gg \simeq
((\mathcal{A}^r_\gg)^!)^\mathrm{opp}$, 
and clearly also an isomorphism $\mathcal{A}_\gg \simeq (\mathcal{A}_\gg^!)^\mathrm{opp}$.  
\end{proof}

For a partition $\mu = (\mu_1, \mu_2, \ldots, \mu_s)$, we set
$\mu^\perp := (s = \# \mu, \# (\mu_1 - 1, \mu_2 - 1, \ldots),\ldots)$, or in terms of Young diagrams, $\mu^\perp$ is the
conjugate partition obtained from $\mu$ by interchanging rows and columns.

\begin{cor}  \label{multiplicity}
Let  $\gg=\sl(\infty)$, and for a pair of partitions $\nu = (\nu^1, \nu^2)$ take $\nu^\perp:=(\nu^1,(\nu^2)^\perp)$. Then
  $\operatorname{dim}\operatorname{Ext}^i_{\Tg}(V_\lambda,V_\mu)$ equals the multiplicity of
$V_{\lambda^\perp}$ in $\operatorname{soc}^i(\tilde{V}_{\mu^\perp})/\operatorname{soc}^{i-1}(\tilde{V}_{\mu^\perp})$, as computed in \cite[Theorem 2.3]{PStyr}.
\end{cor}

\begin{proof} The statement follows from \cite[Theorem 2.10.1]{BGS} applied to
  $\mathcal A^r_\gg$ for sufficiently large $r$. Indeed, this result implies that 
$\operatorname{Ext}_{\mathcal A_\gg}((\mathcal A_\gg)_0,(\mathcal A_\gg)_0)$ is
isomorphic to  $(\mathcal A_\gg^!)^\mathrm{opp}$ as a graded algebra. 
Moreover, the simple $\mathcal A_\gg$-module $X_\lambda$ (which is the image of $V_\lambda$ under the antiequivalence of Corollary~\ref{bigequiv}) is isomorphic to $(\mathcal A_\gg)_0\mathbb Y_\lambda $, where 
$\mathbb Y_\lambda$ is the product of Young projectors  corresponding
to the partitions $\lambda^1$ and $\lambda^2$. This follows immediately
from the fact $\mathbb Y_\lambda$ is a primitive idempotent in
$(\mathcal A_\gg)_0$ and hence also in $\mathcal A_\gg$, see for example \cite[Theorem 54.5]{CR}. 
The projective cover $\tilde{X}_\lambda$ of $X_\lambda$ is isomorphic to $\mathcal A_\gg\mathbb Y_\lambda$. Therefore we have
$$\operatorname{dim}\operatorname{Ext}^i_{\Tg}(V_\lambda,V_\mu)=
\operatorname{dim}\operatorname{Ext}^i_{\mathcal
  A_\gg}(X_\mu,X_\lambda)=
\operatorname{dim}\mathbb Y_\lambda (\mathcal A_\gg^!)^\mathrm{opp}_i \mathbb Y_\mu.$$
 By Corollary \ref{selfduality},
$$\operatorname{dim}\mathbb Y_\lambda 
(\mathcal A_\gg^!)^\mathrm{opp}_i \mathbb Y_\mu=
\operatorname{dim}\mathbb Y_{\lambda^\perp} (\mathcal A_\gg)_i \mathbb
Y_{\mu^\perp}.$$
Furthermore, 
$\operatorname{dim}\mathbb Y_{\lambda^\perp} (\mathcal A_\gg)_i \mathbb
Y_{\mu^\perp}$ equals the multiplicity of $X_{\lambda^\perp}$ in 
$\operatorname{rad}^{i-1}\tilde{X}_{\mu^\perp} / \operatorname{rad}^i\tilde{X}_{\mu^\perp}$  \cite[Theorem 54.15]{CR}), which coincides with the multiplicity of $V_{\lambda^\perp}$ in $\operatorname{soc}^i(\tilde{V}_{\mu^\perp})/\operatorname{soc}^{i-1}(\tilde{V}_{\mu^\perp})$.
\end{proof}

\begin{cor}
The blocks of the category $\T_{\sl(\infty)}$ are parametrized by $\ZZ$.  In particular, 
\begin{enumerate}
\item $V_\lambda$ and $V_\mu$ belong to the block $\T_{\sl(\infty)} (i)$ for $i \in \ZZ$ if and only if $|\lambda^1| - |\lambda^2| = |\mu^1| - | \mu^2| = i$.
\item Two blocks $\T_{\sl(\infty)} (i)$ and $\T_{\sl(\infty)} (j)$ are equivalent if and only if $i = \pm j$. 
\end{enumerate}
\end{cor}

\begin{proof}
\begin{enumerate}
\item The fact that $\tilde{V}_{(\mu^1,\mu^2)}$ is an injective hull of $V_{(\mu^1,\mu^2)}$, together with Theorem 2.3 in \cite{PStyr}, implies that
  $\operatorname{Ext}_{\Tg}^1(V_{(\mu^1,\mu^2)},V_{(\lambda^1,\lambda^2)})\neq 0$
  iff $\mu^1\in(\lambda^1)^+$ and
  $\mu^2\in(\lambda^2)^+$.  More precisely, Theorem 2.3 in \cite{PStyr} computes the multiplicities of the constituents of the socle of $\tilde V_\lambda / V_\lambda$, and a simple module has nonzero $\operatorname{Ext}_{\Tg}^1$ with $V_\lambda$ if and only if it is isomorphic to a submodule of $\tilde V_\lambda / V_\lambda$.  Consider the minimal equivalence relation
  on pairs of partitions for which $(\lambda^1,\lambda^2)$ and $(\mu^1,\mu^2)$ are equivalent whenever $\mu^1\in(\lambda^1)^+$ and $\mu^2\in(\lambda^2)^+$.
  It is a simple exercise to show that then $\lambda=(\lambda^1,\lambda^2)$
  and $\mu=(\mu^1,\mu^2)$ are equivalent if and only if  $|\lambda^1| -  |\lambda^2| = |\mu^1| - | \mu^2|.$ 
The first assertion follows.
\item The functor $( \, \cdot \, )_*$ establishes an equivalence of
  $\T_{\sl(\infty)} (i)$ and $\T_{\sl(\infty)} (-i)$.  To see that
  $\T_{\sl(\infty)} (i)$ and $\T_{\sl(\infty)} (j)$ are inequivalent
  for $i \neq \pm j$, assume without loss of generality that $i>0$, $j
  \geq 0$.  Then the isomorphism classes of simple injective
objects in $\T_{\sl(\infty)} (i)$ are parametrized by the partitions
of $i$, since $\{ V_{(\lambda^1 , 0)} \mid  |\lambda^1| = i \}$
represents the set of isomorphism classes of simple injective objects
in $\T_{\sl(\infty)} (i)$.  As the sets 
$\{ V_{(\lambda^1 , 0)} \mid |\lambda^1| = i \}$ and $\{ V_{(\lambda^1  , 0)} \mid  |\lambda^1| = j\}$ have different cardinalities for $i\neq j$ 
except the case $i=1,j=0$, the assertion follows in other cases. Each of the blocks $\T_{\sl(\infty)} (0)$ and $\T_{\sl(\infty)} (1)$ has a single simple injective module, up to isomorphism. However, $V$ has nontrivial extensions by both $V_{((2),(1))}$ and $V_{((1,1),(1))}$, whereas $\CC$ has a nontrivial extension only by $V_{((1),(1))}$.  This completes the proof.
\end{enumerate}
\end{proof}

Now we proceed to describing the structure of $\mathcal A_\gg$ for
$\gg=\so(\infty)$ and $\sp(\infty)$.  Recall that
$(\mathcal A_\gg)_i^p=\operatorname{Hom}_\gg(T^p,T^{p-2i}).$ and $(\mathcal A_\gg)_0^p=\mathbb C[S_p]$.
Let $S_{p-2}\subset S_p$ denote the stabilizer of
$p$ and $p-1$, and let $S' \subset S_p$ be the subgroup generated by the transposition $(p-1,p)$.

\begin{lemma} We have
\begin{align*}
(\mathcal A_\gg)_1^p &\simeq triv \otimes_{\mathbb C[S']} \mathbb C[S_{p}] &\textrm{for } \gg=\so(\infty) \\
\intertext{and}
(\mathcal A_\gg)_1^p &\simeq sgn \otimes_{\mathbb C[S']} \mathbb C[S_{p}] &\textrm{for } \gg=\sp(\infty).
\end{align*}
In both cases left multiplication by $\mathbb C[S_{p-2}]$ is well defined, as $S'$ centralizes $S_{p-2}$.
\end{lemma}
\begin{proof} Lemma~\ref{contractions} implies that the contraction $\psi_{p-1,p}$ generates
  $(\mathcal A_\gg)_1^p$ as a right $\mathbb C[S_p]$-module.  Then the
  statement follows from the relation
$$\psi_{p-1,p}=\pm \psi_{p-1,p}(p,p-1),$$
where the sign is $+$ for  $\gg=\so(\infty)$ and $-$ for
$\gg=\sp(\infty)$.
\end{proof}
\begin{cor}  Let $\gg=\so(\infty)$ or $\sp(\infty)$. Then
$$(\mathcal A_\gg)_1^{p-2}\otimes_{(\mathcal A_\gg)_0^{p-2}} (\mathcal
  A_\gg)_1^{p}\simeq L_\gg\otimes_{\mathbb C[S]}\mathbb C[S_p],$$
where $S \simeq S_2\times S_2$ is the subgroup generated by $(p,p-1)$ and
  $(p-2,p-3)$ and 
  $$L_\gg=
  \begin{cases}
 triv & \textrm{for }\gg=\so(\infty) \\
 sgn\boxtimes sgn & \textrm{for } \gg =\sp(\infty).
 \end{cases}$$
\end{cor}

To describe $R$, write $R=\bigoplus_p R^{p}$, where 
$R^p\subset (\mathcal A_\gg)_1^{p-2}\otimes_{(\mathcal A_\gg)_0^{p-2}} (\mathcal  A_\gg)_1^{p}$.

We will need the following decompositions of $S_4$-modules:
\begin{align}\label{s4}
triv\otimes_{\mathbb C[S]}\mathbb C[S_4] & =X_{(2,1,1)}\oplus X_{(2,2)}\oplus X_{(4)}, \\
(sgn\boxtimes sgn)\otimes_{\mathbb C[S]}\mathbb C[S_4]& =X_{(3,1)}\oplus X_{(2,2)}\oplus X_{(1,1,1,1)}.
\end{align}

\begin{lemma} Let $S'' \subset S_p$ be the subgroup isomorphic to $S_4$ that fixes $1$, $2$,\dots, $p-4$. Then

\begin{align*}
R^p & \simeq  X_{(2,1,1)}\otimes_{\mathbb C[S'']}\mathbb C[S_p] & \textrm{for } \gg=\so(\infty), \\
\intertext{and}
R^p & \simeq  X_{(3,1)}\otimes_{\mathbb C[S'']}\mathbb C[S_p] &\textrm{for } \gg=\sp(\infty).
\end{align*}
\end{lemma}
\begin{proof} Let us deal with the case of  $\so(\infty)$. 
We consider the following Young projectors  in $S'' \simeq S_4$
$$\mathbb Y_{(2,1,1)}=(1+(p-1,p))(1-(p,p-2)-(p,p-3)-(p-2,p-3)+(p,p-2,p-3)+(p,p-3,p-2)),$$
$$\mathbb Y_{(2,2)}=(1+(p,p-1))(1+(p-2,p-3))(1-(p-2,p))(1-(p-1,p-3)),$$
and $$\mathbb Y_{(4)}=\sum_{s\in S''}s.$$
By Equation \eqref{s4} we have
$$R^p\subset (\mathcal A_\gg)_1^{p-2}\otimes_{(\mathcal
  A_\gg)_0^{p-2}} (\mathcal  A_\gg)_1^{p}=\mathbb Y_{(2,1,1)}\mathbb C[S_p]\oplus 
\mathbb Y_{(2,2)}\mathbb C[S_p]\oplus \mathbb Y_{(4)}\mathbb C[S_p].$$
By direct inspection one can check that
\begin{align*}
\psi_{p-3,p-2}\psi_{p-1,p}\mathbb Y_{(2,1,1)}&=0,\\
\psi_{p-3,p-2}\psi_{p-1,p}\mathbb Y_{(2,2)}&=2\psi_{p-3,p-2}\psi_{p-1,p}-2\psi_{p-3,p}\psi_{p-1,p-2},\\
\psi_{p-3,p-2}\psi_{p-1,p}\mathbb Y_{(4)}&=4\psi_{p-3,p-2}\psi_{p-1,p}.
\end{align*}
The statement follows for $\so(\infty)$.

We leave the case of $\sp(\infty)$ to the reader. 
\end{proof}

\begin{cor} $\mathcal A_{\sp(\infty)}\simeq \mathcal A_{\so(\infty)}$.
\end{cor}
\begin{proof} We use the automorphism $\sigma$ of $\mathbb C[S_p]$ which sends $s$
  to $sgn(s)s$.
\end{proof}

\begin{cor}\label{isomosp}
The categories $\T_{o(\infty)}$ and $\T_{sp(\infty)}$ are equivalent.
\end{cor}

\begin{prop} $\T_{o(\infty)}$ and $\T_{sp(\infty)}$ have two
  inequivalent blocks  $\T_{\gg}^{ev}$ and  $\T_{\gg}^{odd}$
  generated by all $V_\lambda$ with $|\lambda|$ even and odd, respectively.
\end{prop}
\begin{proof} Due to the previous corollary it suffices to consider the
  case $\gg=\so(\infty)$. As follows from \cite{PStyr},
  $\operatorname{Ext}_{\Tg}^1(V_\mu,V_\lambda)\neq 0$ if and only if $\mu \in \lambda^{++}$, 
  where
\begin{align*}
\lambda^{++} := \{ \textrm{partitions } \lambda' \mid & \lambda_i \leq \lambda'_i \textrm{ for all }i  \textrm{, }
  |\lambda'| = |\lambda| + 2  \textrm{, } \\ 
 & \lambda'_j \neq \lambda_j \textrm{ and } \lambda'_k \neq \lambda_k \textrm{ for } j \neq k \textrm{ implies } \lambda_j \neq \lambda_k \}.
\end{align*} 
Note that the partitions in $ \lambda^{++}$ are those which arise from $\lambda$ via the Pieri rule for tensoring with $S^2(V)$.
Consider the minimal equivalence relation on
  partitions for which $\lambda$ and $\mu$ are equivalent whenever $\mu \in \lambda^{++}$. One can check that there are exactly two equivalence
  classes which are determined by the parity of $|\lambda|$. 

To show that $\T_{\gg}^{ev}$ and  $\T_{\gg}^{odd}$ are not
equivalent observe that  all simple injective modules in $\Tg$ correspond to
partitions $\mu$ with $\mu_1=\cdots=\mu_s=1$, or equivalently are isomorphic
to the exterior powers $\Lambda^s(V)$ of the standard module. If $s\geq 1$
then $\Lambda^s(V)$ has nontrivial extensions by two non-isomorphic simple modules, namely
$V_{(3,1, \ldots, 1)}$ and $V_{(2,1,1, \ldots, 1)}$.  The trivial module on the other hand has a nontrivial extension by only $S^2(V) = V_{(2)}$.  Therefore $\T_{\gg}^{ev}$ contains a simple injective module admitting a nontrivial extension with only one simple module, whereas $\T_{\gg}^{odd}$ does not contain such a simple injective module.
\end{proof}

\end{document}